%% file: TWandVolume.tex
\documentclass[11pt]{amsart}

\usepackage{amsmath, amssymb, amsthm, amscd}
\usepackage{url}
\usepackage{graphicx}
\usepackage[hidelinks,pagebackref,pdftex]{hyperref}
\usepackage{color}
\usepackage{import}
\usepackage[margin=3cm]{geometry}

\usepackage{marginnote}
\long\def\@savemarbox#1#2{\global\setbox#1\vtop{\hsize\marginparwidth 
  \@parboxrestore\tiny\raggedright #2}}
\renewcommand*{\backref}[1]{}
\renewcommand*{\backrefalt}[4]{
  \ifcase #1
  [No citations.]
  \or [#2]
  \else [#2]
  \fi }

\AtBeginDocument{%
   \def\MR#1{}
}

\numberwithin{equation}{section}
\theoremstyle{plain}
\newtheorem{theorem}[equation]{Theorem}

\newtheorem{lemma}[equation]{Lemma}

\newtheorem{corollary}[equation]{Corollary}
\newtheorem{proposition}[equation]{Proposition}

\newtheorem*{namedtheorem}{\theoremname}
\newcommand{\theoremname}{testing}

\theoremstyle{definition}
\newtheorem{definition}[equation]{Definition}
\newtheorem{example}[equation]{Example}
\newtheorem{remark}[equation]{Remark}

\newcommand{\from}{\colon} 
\newcommand{\HH}{{\mathbb{H}}}
\newcommand{\RR}{{\mathbb{R}}}

\newcommand{\RP}{\mathbb{R}P}

\newcommand{\G}{\mathcal{G}}
\newcommand{\Hgraph}{\mathcal{H}}
\newcommand{\calT}{\mathfrak{T}}
\newcommand{\M}{M}
\newcommand{\Mmu}{\M^{\geq \mu}}

\newcommand{\vol}{\operatorname{vol}}

\newcommand{\bdy}{\partial}

\newcommand{\tw}{\operatorname{tw}}

\newcommand{\dist}{\mathrm{dist}}

\newcommand{\cng}{\operatorname{cng}}

\newcommand{\poly}{\mathrm{poly}}
\newcommand{\oracle}{\mathrm{Or}}
\newcommand{\refthm}[1]{Theorem~\ref{Thm:#1}}
\newcommand{\reflem}[1]{Lemma~\ref{Lem:#1}}
\newcommand{\refprop}[1]{Proposition~\ref{Prop:#1}}
\newcommand{\refcor}[1]{Corollary~\ref{Cor:#1}}
\newcommand{\refrem}[1]{Remark~\ref{Rem:#1}}

\newcommand{\refdef}[1]{Definition~\ref{Def:#1}}
\newcommand{\refsec}[1]{Section~\ref{Sec:#1}}
\newcommand{\reffig}[1]{Figure~\ref{Fig:#1}}

\title{Treewidth, crushing, and hyperbolic volume}
\author{Cl\'ement Maria}
\author{Jessica S. Purcell}

\begin{document}

\begin{abstract}
  The treewidth of a 3-manifold triangulation plays an important role in algorithmic 3-manifold theory, and so it is useful to find bounds on the treewidth in terms of other properties of the manifold. In this paper, we prove that there exists a universal constant $c$ such that any closed hyperbolic 3-manifold admits a triangulation of treewidth at most the product of $c$ and the volume. The converse is not true: we show there exists a sequence of hyperbolic 3-manifolds of bounded treewidth but volume approaching infinity.  Along the way, we prove that crushing a normal surface in a triangulation does not increase the carving-width, and hence crushing any number of normal surfaces in a triangulation affects treewidth by at most a constant multiple.
\end{abstract}

\maketitle

\section{Introduction}

This paper concerns invariants of 3-manifolds that are of interest both in geometry and in computational topology.
For computational purposes, 3-manifolds are often expressed by a \emph{triangulation}, that is by gluing a collection of tetrahedra. For example, this is true for 3-manifold software {\tt SnapPea}, developed by Weeks in the early 1980s, now maintained and distributed as {\tt SnapPy}~\cite{snappy}, and for {\tt Regina}, developed by Burton~\cite{burton04-regina}. These programs have been influential in the development of 3-manifold geometry and topology. Computational topology considers the running time of algorithms. For an algorithm that takes a triangulation as input, the running time frequently depends on some measure of the ``simplicity'' 
of the triangulation. Note, however, that a 3-manifold can have many different triangulations. Therefore, it is important to produce triangulations that are as simple as possible. Here, we will evaluate the simplicity of a triangulation by its \emph{treewidth}.

The treewidth of a triangulation is a measure of the sparsity of the gluing relations between tetrahedra; see \refdef{Treewidth}. It was first developed in graph theory~\cite{robertson86-algorithmic}, then adapted to 3-manifold triangulations. In recent years, several algorithms have been developed that are highly efficient for triangulations with low treewidth~~\cite{DBLP:conf/icalp/BurtonMS15,DBLP:conf/soda/BurtonS13}, and so we would like to find triangulations of 3-manifolds with treewidth bounded in terms of well-understood properties of the manifold. 

One property is geometry. By the geometrisation theorem proved by Perelman (\cite{perelman02, perelman03}, or see~\cite{kleiner08-perelman}), every closed orientable 3-manifold decomposes into geometric pieces, and the hyperbolic pieces are among the most prevalent and least understood. 
If a closed 3-manifold admits a hyperbolic structure, then that structure is a topological invariant of the manifold \cite{Mostow}, and so it is natural to ask if the hyperbolic geometric properties of the manifold can bound treewidth of a triangulation. 

A hyperbolic invariant that has received much attention is the hyperbolic volume. By work of J{\o}rgensen and Thurston, a hyperbolic 3-manifold $M$ that has a lower bound on injectivity radius admits a triangulation with $O(\vol(M))$ tetrahedra (\cite{Thurston:notes}, see also \cite{KobayashiRieck}). However, if we put no restrictions on injectivity radius, then no such result holds: for a sufficiently large constant $C>0$, there are infinitely many closed hyperbolic 3-manifolds with volume bounded above by $C$, and therefore a finite number of tetrahedra cannot triangulate them all. For example, such manifolds are obtained by Dehn filling a hyperbolic manifold with finite volume, using the fact that volume decreases under Dehn filling (\cite{Thurston:notes}, or see \cite[Chapter~E]{BenedettiPetronioHyperbolicGeom}). 
Nevertheless, we prove in this paper that any hyperbolic 3-manifold with bounded volume admits a triangulation with bounded treewidth.

\begin{theorem}\label{Thm:TreeWidth}
There exists a universal constant $c>0$ such that a hyperbolic 3-manifold $M$ with volume $\vol(M)$ admits a triangulation with treewidth at most $c \cdot \vol(M)$.
\end{theorem}

In computer science, parameterized complexity classifies computational difficulty in terms of multiple parameters as input. Some problems that are known to require superpolynomial time in terms of the input alone, under standard computational complexity assumptions, can be solved by algorithms that are exponential in one fixed parameter, but only polynomial in the size of another. Thus they can be solved efficiently for low values of the first parameter. An important example of this from graph theory is Courcelle's theorem, which states that many graph theory problems can be decided in linear time in the treewidth of the graph~\cite{Courcelle}. Recently, Courcelle's theorem has been adapted to 3-manifold topology by Burton and Downey~\cite{DBLP:journals/jct/BurtonD17}.
Along with the rich theory of parameterized algorithms and standard dynamic programming techniques, this has led to the development of several algorithms 
in 3-manifold topology that are both theoretically and practically efficient, provided the input triangulation has small treewidth. Some of these parameterized algorithms have been implemented in the 3-manifold software {\tt Regina}~\cite{burton04-regina}, and have led to significant improvement in practical computations.

In practice, the treewidth parameter is strongly dependent on the triangulation chosen for representing a manifold, and obtaining low treewidth triangulations can be difficult; see~\cite{DBLP:conf/soda/MariaS17} for a discussion. Unfortunately, a manifold that has a simple topological or geometric description can often be represented by a triangulation that has extremely large treewidth, with no obvious combinatorial simplifications. Therefore it is important to identify triangulations of a manifold whose treewidth is bounded by topological or geometric properties of the manifold, as in \refthm{TreeWidth}.

\smallskip

We also consider the converse to \refthm{TreeWidth}, and show it does not hold. In \refthm{BddTW}, we show that there exists a sequence of closed hyperbolic manifolds with bounded treewidth and volume approaching infinity. Thus, while volume gives an upper bound on treewidth, it does not give a lower bound.

On the other hand, recent work of Husz{\'a}r, Spreer, and Wagner~\cite{HuszarSpreerWagner} implies that there is a sequence of 3-manifolds whose treewidth approaches infinity. A corollary of our result is that any such examples that are hyperbolic have volume also approaching infinity.
In fact, one family of examples is the family of small manifolds with large genus developed by Agol~\cite{Agol:Small}. For the $n$-th manifold in this family, combining work of~\cite{Agol:Small} with~\cite{HuszarSpreerWagner}, the treewidth is at least $n/2$, but the volume is of order $O(n^2)$.  It would be interesting to find a family for which volume and treewidth grow proportionally, to determine whether there is hope of improving \refthm{TreeWidth}.


\subsection{Crushing, carving-width, and treewidth}

The proofs of \refthm{TreeWidth} and \refthm{BddTW} modify triangulations using the crushing procedure developed by Jaco and Rubinstein~\cite{JacoRubinstein:0Eff} and simplified by Burton~\cite{Burton:Crushing}. In order to prove the theorems, we show in \refcor{carving-width} that crushing does not affect a different measure of
the sparsity of gluing relations of the triangulation,
namely the \emph{carving-width}; see \refdef{Carvingwidth}. 
For a 3-manifold triangulation, it is known that the carving-width is at least 2/3 the treewidth and at most four times the treewidth; see \refthm{boundcngtw}. Thus crushing any finite number of times affects treewidth by at most a multiplicative constant.

Note that a usual pipeline for practical computation in 3-manifold topology consists of first, simplifying a triangulation using efficient implementations of the crushing procedure, and then running computations. Corollaries \ref{Cor:carving-width} and \ref{Cor:CrushingTreewidth} guarantee that this approach does not affect the computational complexity of parameterized algorithms using the carving-width or treewidth as a parameter, such as~\cite{DBLP:journals/jct/BurtonD17,DBLP:conf/icalp/BurtonMS15,DBLP:conf/soda/BurtonS13}. 
Thus these results on crushing and computational complexity are important, and likely of independent interest.

\subsection{Outline}
In \refsec{Treewidth} we review results on treewidth and carving-width. Then in \refsec{Crushing}, we prove that the Jaco-Rubinstein crushing procedure does not increase the carving-width of a triangulation. To prove \refthm{TreeWidth}, we use the fact that there is a universal Margulis constant $\mu$ such that any hyperbolic manifold $M$ can be obtained from its $\mu$-thick part $\Mmu$ by hyperbolic Dehn filling; see \refsec{HyperbolicReview}. The proof begins by taking a geodesic triangulation of $\Mmu$  with $O(\vol(M))$ \emph{fat} tetrahedra, i.e.\ tetrahedra with volume bounded from below (\refsec{geometry}). Next, we describe how to proceed to the Dehn filling without increasing the treewidth of the whole triangulation (\refsec{combinatorics}). We consequently obtain a triangulation with treewidth $O(\vol(M))$, and describe an explicit algorithm to construct it. The number of tetrahedra of the triangulation depends solely on the hyperbolic volume $\vol(M)$ and the slopes of the hyperbolic Dehn surgeries.

Finally, in \refsec{cst_tw}, we prove there exists a family of closed hyperbolic 3-manifolds with unbounded volume that admits a triangulation with constant treewidth. 

\section{Triangulations, carving-width, and treewidth}\label{Sec:Treewidth}

In this section, we define several necessary terms and fix notation. 

Let $M$ be a closed 3-manifold. A {\em cell-decomposition} of $M$ is a pairwise-disjoint collection of $n$ oriented, compact, convex linear 3-cells $\Delta_1,\ldots,\Delta_n$ equipped with affine maps that identify (or ``glue together'') their faces in pairs, so that the underlying topological space is homeomorphic to $M$. The {\em dual graph} of a cell decomposition is the graph, with multiple arcs and loops, having a node for every 3-cell $\Delta_i$, and an arc $(\Delta_i,\Delta_j)$ for every face gluing between 3-cells $\Delta_i$ and $\Delta_j$.

A \emph{generalised triangulation} of $M$ is a cell-decomposition where all 3-cells are abstract tetrahedra. Its dual graph is naturally a 4-valent graph, corresponding to gluings of triangular faces. 
Generalised triangulations are widely used across major 3-manifold software packages, and they allow the representation of a rich variety of 3-manifolds using very few tetrahedra.

We also encounter \emph{ideal triangulations} in this work, which are triangulations from which vertices have been removed. A removed vertex is called an \emph{ideal vertex}. 


\begin{remark}\label{Rem:Notation}
  We will be discussing both graphs and triangulations. We will refer to \emph{nodes} and \emph{arcs} of graphs, to clearly distinguish these from the \emph{vertices} and \emph{edges} of triangulations.
\end{remark}
  
The \emph{carving-width}, also known as \emph{congestion}, is a graph parameter introduced by Seymour and Thomas~\cite{Seymour1994}.

\begin{definition}\label{Def:Carvingwidth}
Let $\G$ be a graph, possibly with loops and multiple arcs between nodes, defined on $n$ nodes, and let $T$ be an unrooted binary tree, with all internal nodes of degree 3, and with $n$ leaves.
An {\em embedding} $\pi$ of $\G$ into $T$ is an injective mapping from the nodes of $\G$ to the leaves of $T$. To every pair of endpoints $(u, v)$ of an arc in $\G$, there corresponds a unique path $p(\pi(u),\pi(v))$ in $T$, connecting leaves $\pi(u)$ and $\pi(v)$. Define the {\em congestion} of an embedding $\pi$ to be:
\[
\cng(\pi) = \max_{a \ \text{arc of} \ T} \left| \{ (u,v) \ \text{in} \ \G : p(\pi(u),\pi(v)) \ \text{contains} \ a \}\right|,
\]
where note we count a multiple arc only once in the formula. 
Here $|\cdot|$ denotes the number of elements in a set.

The {\em carving-width} $\cng(\G)$ of a graph $\G$ is the minimal congestion over all its embeddings into binary trees.
The \emph{carving-width of a cell-decomposition} $\cng(\calT)$ is the carving-width of its dual graph. Finally, we define the \emph{carving-width of a 3-manifold} $M$, denoted by $\cng(M)$, to be the minimal carving-width over all its generalised triangulations.
\end{definition}

In this article, our definition of congestion differs from the literature by counting a multiple arc only once in the tree embedding. However, for a graph dual to a triangulation of a 3-manifold, this only affects the carving-width by a constant multiple, since all dual graphs of triangulations have constant maximal degree four. Also, note that a loop arc $(u,u)$ leads to paths $p(\pi(u),\pi(u))$ of length 0 in a tree embedding, and can be disregarded when computing carving-width. 

We give an example of a tree embedding in \reffig{example}. The following additional example will be important to our applications. 

\begin{example}\label{Example:LST}
Let $M$ be a solid torus. We describe a well-known triangulation of $M$, discussed in detail by Jaco and Rubinstein~\cite{JacoRubinstein:LayeredTriang}, called a \emph{layered} triangulation or \emph{layered solid torus}.
At the core is a triangle with two sides identified to form a M\"obius band. The first tetrahedron is glued such that two of its faces glue to the core triangle, one on either side of the triangle. By gluing correctly, the result is homeomorphic to a solid torus and has boundary consisting of exactly two triangles; see~\cite{JacoRubinstein:LayeredTriang} for details. Additional tetrahedra may now be added inductively. At each step, a single tetrahedron is attached to the existing triangulation such that two of its faces are glued to the two boundary faces, covering a boundary edge (which then becomes an interior edge). The result is a solid torus with two triangular boundary faces. For any fixed slope on the torus, there exists a layered solid torus for which that slope is the meridian, i.e.\ bounds a disc; see for example~\cite[Theorem~4.1]{JacoSedgwick}.

Consider now the gluing graph of a layered solid torus. The tetrahedron at the core with its two faces identified gives a loop at the corresponding node. The other two faces are identified to a single tetrahedron, giving two arcs to the next node. Any additional nodes are connected by two nodes to the previous node, and two nodes to the next node. This forms a simple daisy chain. See \reffig{daisy}, left.
\end{example}

\begin{figure}
\centering
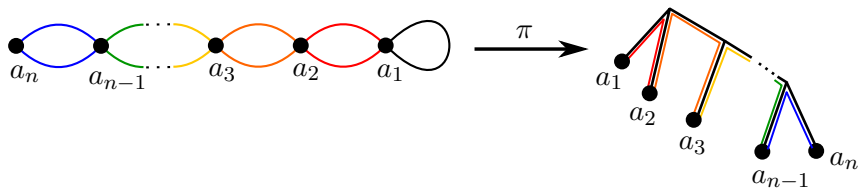
\caption{On the left is the daisy chain graph. Shown is an embedding into a tree, on the right, indicating that the carving-width is at most two.}
\label{Fig:daisy}
\end{figure}

Note if there are only two nodes in the daisy chain graph, the carving-width is $1$. In the more general case we obtain the following.

\begin{lemma}\label{Lem:DaisyChain}
  The daisy chain graph with $n\geq 3$ nodes, arising as the dual of a layered solid torus, has carving-width two.
\end{lemma}

\begin{proof}
The carving-width is at least two, because the carving-width of a graph is at least the maximal degree of a node after identifying multiple arcs, which is two for the daisy chain.

We now show that the carving width is at most two. 
Number the nodes of the daisy chain linearly as in \reffig{daisy}, by $a_1, a_2, \dots, a_{n-1}, a_n$, where the node $a_1$ has a loop and two arcs running to $a_2$, the node $a_2$ has an additional two arcs running to $a_3$, etc., and $a_n$ is 2-valent (two of its faces lie on the boundary of the layered solid torus).

Form an unrooted binary tree with $n$ nodes as follows. Start with the unrooted binary tree $T_3$ with three leaves and a single 3-valent node. Label nodes $b_1$, $b_2$, $b_3$ in a cyclic manner. Note the two paths from $b_1$ to $b_2$ and from $b_2$ to $b_3$ run over the arc connecting the leaf $b_2$ exactly twice, and other arcs exactly once. Now inductively increase the size of the tree until it has $n$ leaves. Given an unrooted tree $T_k$ with $k$ leaves labelled $b_1, \dots, b_k$ in a cyclic manner, form $T_{k+1}$ by attaching two new leaves to the $k$-th leaf, making it a 3-valent node, and label all leaves as in $T_k$ except the two new leaves, labelled $b_k$ and $b_{k+1}$, such that the labelling on $T_{k+1}$ is still cyclic. See \reffig{daisy}, right.

Note that paths from $b_i$ to $b_{i+1}$ for $i<k-1$ are identical in $T_k$ and $T_{k+1}$. For $i=k-1$ in $T_{k+1}$, the path from $b_{k-1}$ to $b_k$ runs over the arc connecting the leaf $b_{k-1}$, along the arc that connects the new 3-valent node in $T_{k+1}$, and then along the arc connecting the leaf $b_k$. The path from $b_k$ to $b_{k+1}$ runs over the two new arcs. Thus by induction, all paths from $b_i$ to $b_{i+1}$, $1\leq i\leq k$, run over arcs connecting leaves $b_2, \dots, b_k$ exactly twice, and all other arcs exactly once. Continue until $k+1=n$, and map each $a_i$ to $b_i$ in $T_n$. Thus the carving-width is at most two. 
\end{proof}

A fundamental property of carving-width is that it decreases when taking {\em immersions}. 

\begin{definition}\label{Def:Immersion}
  Let $\G$ be a graph, with adjacent arcs $(u,v)$ and $(v,w)$. A \emph{lifting} of $uvw$ consists of removing all arcs $(u,v)$ and $(v,w)$ from $\G$, and adding arc $(u,w)$. An \emph{immersion} of $\G$ is a graph $\Hgraph$ that can be obtained from $\G$ by a sequence of liftings, and arc and node removals.

  Equivalently, $\Hgraph$ is an immersion of $\G$ if there exists a mapping of the nodes of $\Hgraph$ to the nodes of $\G$ where every arc $(u,v)$ is sent to a path from $\pi(u)$ to $\pi(v)$ in $\G$ such that distinct arcs in $\Hgraph$ lead to arc-disjoint paths in $\G$.
\end{definition}

The following is standard and follows from the definitions.

\begin{lemma}\label{Lem:Immersion}
If $\Hgraph$ is an immersion of $\G$, then
\[\cng(\G) \geq \cng(\Hgraph).\]
\end{lemma}

\begin{proof}
The nodes of $\Hgraph$ are a subset of the nodes of $\G$, so any embedding of $\G$ into a tree $T$ restricts to an embedding of $\Hgraph$ into the same tree. Form $T_{\Hgraph}$ from $T$ by removing leaves of $T$ that are not the image of nodes of $\Hgraph$, and viewing arcs adjacent to remaining 2-valent nodes as a single arc. Then paths in $T_{\Hgraph}$ between nodes coming from $\Hgraph$ are obtained by taking paths in $T$ between nodes of $\Hgraph$ and removing leaves and 2-valent nodes.

If $\Hgraph$ differs from $\G$ by a lifting of $uvw$, then $\Hgraph$ contains an arc $(u,w)$ and no arcs $(u,v)$ and $(v,w)$, while $\G$ contains arcs $(u,v)$ and $(v,w)$. Note that the unique path in $T$ from $\pi(u)$ to $\pi(w)$ can be obtained by taking the union of paths from $\pi(u)$ to $\pi(v)$ and from $\pi(v)$ to $\pi(w)$ and removing all arcs traversed twice. Thus the arcs in $(\pi(u),\pi(w))$ in $T$ form a subset of those in the two paths $(\pi(u),\pi(v))$ and $(\pi(v),\pi(w))$. 
It follows that the congestion of $\Hgraph$ is at most that of $\G$. 

Finally, for any node and arc removal to convert $\G$ to $\Hgraph$, the corresponding paths will be removed from $T$, so the number of paths $p(\pi(u),\pi(v))$ running over a fixed arc $a$ in $T$, and hence $T_{\Hgraph}$, will also decrease.
\end{proof}

The carving-width of a graph is closely related to \emph{treewidth}, which plays a major role in combinatorial algorithms. The \emph{treewidth} of a graph was introduced by Robertson and Seymour~\cite{robertson86-algorithmic}, and is defined as follows.

\begin{definition}\label{Def:Treewidth}
Let $\G$ be a graph with loops and multiple arcs. A {\em tree decomposition} $(X, \{B_\tau\})$ of $\G$ consists of a tree $X$ and {\em bags} $B_\tau$ of nodes of $\G$ for each node $\tau$ of $X$, for which:
\begin{enumerate}
\item each node $u$ in $\G$ belongs to some bag $B_\tau$;
\item for every arc $(u,v)$ in $\G$, there exists a bag $B_\tau$ containing both $u$ and $v$;
\item for every node $u$ in $\G$, the bags containing $u$ form a connected subtree of $X$.
\end{enumerate}
The \emph{width} of this tree decomposition is defined as $\max_{\tau \in X} |B_\tau|-1$. The \emph{treewidth of $\G$}, denoted $\tw(\G)$, is the smallest width of any tree decomposition of $\G$. 
\end{definition}

Similarly, the treewidth of a cell-decomposition is the treewidth of its dual graph, and the treewidth of a 3-manifold is the minimal treewidth over all of its generalised triangulations.

\reffig{example} shows the dual graph of a $9$-tetrahedra triangulation of a $3$-manifold, along with a possible tree decomposition. The largest bags have size three, and so the width of this tree decomposition is $3-1=2$. 

\begin{figure}[tb]
\centering
\includegraphics{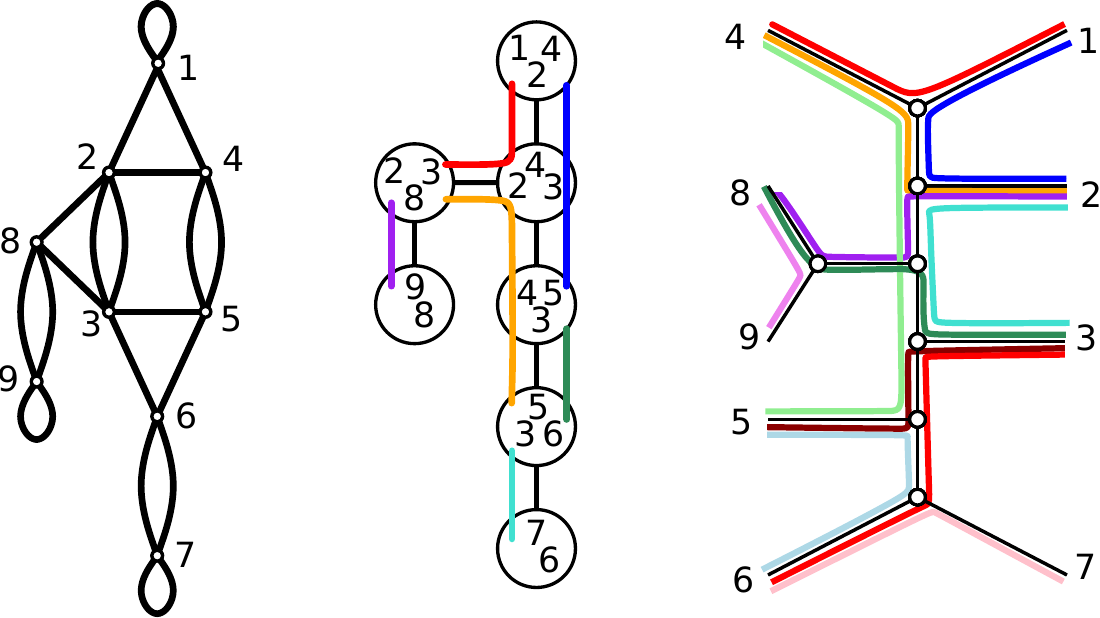}
\caption{The dual graph of a $3$-manifold triangulation (left), a tree decomposition of width $2$ (centre), and a tree embedding of width $4$ (right).}
\label{Fig:example}
\end{figure}

Finally, treewidth and carving-width are closely related, and enjoy similar properties. First, they only differ by a constant multiplicative factor:

\begin{theorem}[Theorem~1 of \cite{DBLP:journals/jct/Bienstock90}]
\label{Thm:boundcngtw}
Let $\G$ be a graph of maximal degree $d$. Then,
\[
  \frac{2}{3} (\tw(\G) +1) \leq \cng(\G) \leq d (\tw(\G) +1).
\] 
\end{theorem}

Note that for dual graphs of generalised triangulations, the degree of every node is at most four, and treewidth and carving-width consequently differ by a small multiplicative constant.

The decision problem associated to computing the treewidth or carving-width of a graph is NP-complete~\cite{Arnborg:1987:CFE:37170.37183,Seymour1994}. However, both treewidth and carving-width, together with an optimal tree decomposition or embedding into a tree, can be computed in time $O(f(k) \cdot n)$ on graphs with $n$ nodes and treewidth/carving-width at most $k$~\cite{DBLP:journals/siamcomp/Bodlaender96,10.1007/3-540-40996-3_17}.

In the following, we use carving-width because of its favourable properties, and we connect it to the more widely used treewidth.

\section{Crushing triangulations does not increase carving-width}
\label{Sec:Crushing}

This section focuses on compact manifolds with or without boundary, which are the main object of study of this article. However, all results cited and introduced extend naturally to ideal triangulations.

\emph{Crushing} of triangulations is a fundamental technique introduced by Jaco and Rubinstein~\cite{JacoRubinstein:0Eff} to simplify 3-manifold triangulations. Let $\calT$ be a generalised triangulation of a 3-manifold $M$. A normal surface $S$ in $\calT$ is a properly embedded surface in $\calT$ that meets each tetrahedron in a (possibly empty) collection of curvilinear triangles and quads, as illustrated in \reffig{Cut}. A \emph{trivial surface} is a normal surface made only of triangles; it always triangulates the link of a vertex. Finally, a \emph{0-efficient triangulation} is a triangulation $\calT$ that either
\begin{itemize}
\item contains no non-trivial normal sphere if $\calT$ is closed or ideal, or
\item contains no non-trivial normal disk, if $\calT$ is bounded.
\end{itemize}

\begin{figure}
  \centering
  \includegraphics[width=5cm]{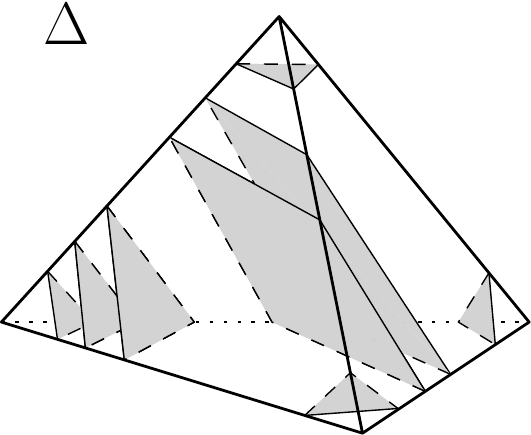}
  \caption{Tetrahedron cut by a normal surface. The intersection is a collection of disjoint normal disks (quads and triangles).}
  \label{Fig:Cut}
\end{figure}

Crushing was introduced in~\cite{JacoRubinstein:0Eff} (see also~\cite{Burton:Crushing}) as a means to construct 0-efficient triangulations.

\begin{definition}\label{Def:Crushing}
Let $S$ be a normal surface in a triangulation $\calT$. \emph{Crushing} the triangulation along $S$ consists of the following three steps:
\begin{enumerate}
\item \textbf{Cut} $\calT$ open along $S$, leading to a cell-decomposition with various cell types, presented in \reffig{Collapse}.
\item \textbf{Collapse} each copy of $S$ to a point, using the quotient topology. This gives four types of cells: tetrahedra, \emph{3-sided footballs}, \emph{4-sided footballs}, and \emph{triangular purses} (all illustrated in \reffig{Collapse}).
\item \textbf{Flatten} all non-tetrahedra cells to obtain a triangulation, i.e.\ flatten footballs into edges, and triangular purses into triangles as in \reffig{Collapse}. 
\end{enumerate}
Conclude by separating tetrahedra joined by pinched vertices and edges.
\end{definition}

Following~\cite[Lemma 3]{Burton:Crushing}, the flattening step can be performed iteratively, one non-tetrahedron cell at a time. In particular, flattening a football or a triangular purse induces the flattening of bigonal faces in the adjacent cells, hence creating temporary cells of new types: triangular purses with one or two flattened bigons (also known as bigonal pyramid and triangular pillows, respectively), and 2-sided footballs (also known as bigonal pillows).

\begin{figure}
  \centering
  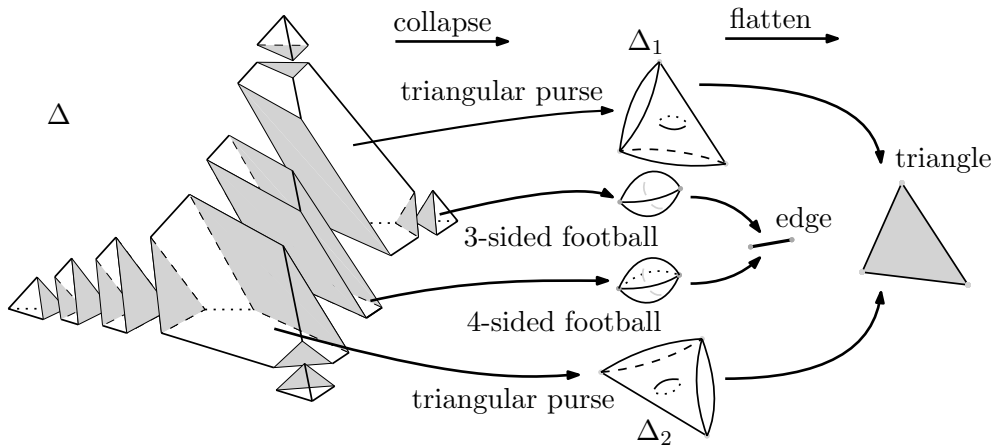
  \caption{Cut out tetrahedron containing quads, collapsing of the cells and flattening. The collapsing produces two triangular purses, and a collection of 3 and 4-sided footballs. Footballs are flattened into edges, and triangular purses into triangles.}
  \label{Fig:Collapse}
\end{figure}

In particular, we use the following property. 

\begin{theorem}[Jaco-Rubinstein~\cite{JacoRubinstein:0Eff}, see also Burton~\cite{Burton:Crushing}]
Let $\calT$ be a generalised triangulation of a closed or bounded 3-manifold $M$. There is an algorithm to construct a finite family of triangulations $\calT_1, \ldots, \calT_n$ triangulating manifolds $M_1, \ldots, M_n$, such that $M = M_1 \# \ldots \# M_n$, and each $\calT_i$ is either 0-efficient, or can be shown to be a triangulation of $S^3$, $S^2 \times S^1$, $\RP^3$ or $L(3,1)$. The algorithm consists of finding normal spheres and disks in the original triangulation, and crushing them.  
\label{Thm:JRCrushing}
\end{theorem}

Jaco and Rubinstein prove that 0-efficient triangulations of a closed irreducible manifold have one vertex, and 0-efficient triangulations of a bounded irreducible $\partial$-irreducible manifold (without 2-sphere boundary components\footnote{This is a technicality that only rules out the simple case where the irreducible manifold is a 3-cell.}) have all vertices in the boundary, with exactly one vertex per boundary component.

We now introduce the main combinatorial result of this article, namely that crushing does not increase the carving-width. We use this result repeatedly as a tool to manipulate hyperbolic manifolds in the latter sections.

\begin{theorem}\label{Thm:crushing}
Let $\calT$ be a generalised (or ideal, or bounded) triangulation of a 3-manifold $M$, and $S$ a normal surface in $\calT$. Let $\calT^*$ be the triangulation obtained after crushing $S$. Then the dual graph of $\calT^*$ is an immersion of the dual graph of $\calT$.
\end{theorem}

\begin{proof}
We track the evolution of the cell decomposition of the triangulation under the three steps (cut, collapse, flatten) of crushing $S$ in $\calT$, in order to describe the change to its dual graph. Let $\G$ be the dual graph of $\calT$.

Cut $\calT$ along $S$. The result is a collection of cells; each tetrahedron that meets $S$ in $\calT$ is split into cells across normal discs of $S$ as in Figures \ref{Fig:Cut} and~\ref{Fig:Collapse}.

Now collapse each normal disc of $S$ to a point, using the quotient topology, to obtain the cell complex $\calT'$. This operation splits every tetrahedron $\Delta$ of $\calT$ into a collection $C_\Delta = \{\Delta_0, \ldots, \Delta_n\}$ of cells, where $n$ is the number of normal disks in $\Delta \cap S$. The cells are of four types: tetrahedra, $3$-sided footballs, $4$-sided footballs, and triangular purses; see \reffig{Collapse}.

Note that if $\Delta \cap S$ contains no quad, then $C_\Delta$ is made of exactly one (central) tetrahedron, and a possibly empty collection of $3$-sided footballs. If $\Delta \cap S$ contains quads, then $C_\Delta$ is made of two triangular purses, and a possibly empty collection of $3$-sided and $4$-sided footballs. 

Now flatten. We obtain a generalised triangulation $\calT^*$. The $3$ and $4$-sided footballs become edges, and thus have no dual nodes or arcs in $\calT^*$. A tetrahedron is not flattened, thus the dual graph of $\calT^*$ has one node corresponding to $\Delta\cap S$ containing no quads. Two triangular purses flatten to triangles, thus removing the node corresponding to $\Delta$.

Note if we perform this process one tetrahedron at a time, adjusting the dual graph one node at a time, then each node corresponding to $\Delta$ that does not meet a quad will be replaced by a node corresponding to a tetrahedron of the crushing; the arcs from this node will run to the same nodes as before the replacement, since faces of the new tetrahedron are still glued to faces of adjacent tetrahedra. For each node corresponding to $\Delta$ that meets a quad, the node is removed, and two liftings are performed: when each triangular purse is flattened, it identifies two triangular faces together and removes a node. Thus faces of adjacent tetrahedra become glued through this triangle. The result is a lifting. See \reffig{Lifting}.

Perform this process for each tetrahedron. Then separate tetrahedra joined by pinched vertices and edges, which does not affect the dual graph. We see that the final result is an immersion. 
\end{proof}

\begin{figure}
  \centering
  \includegraphics[width=12cm]{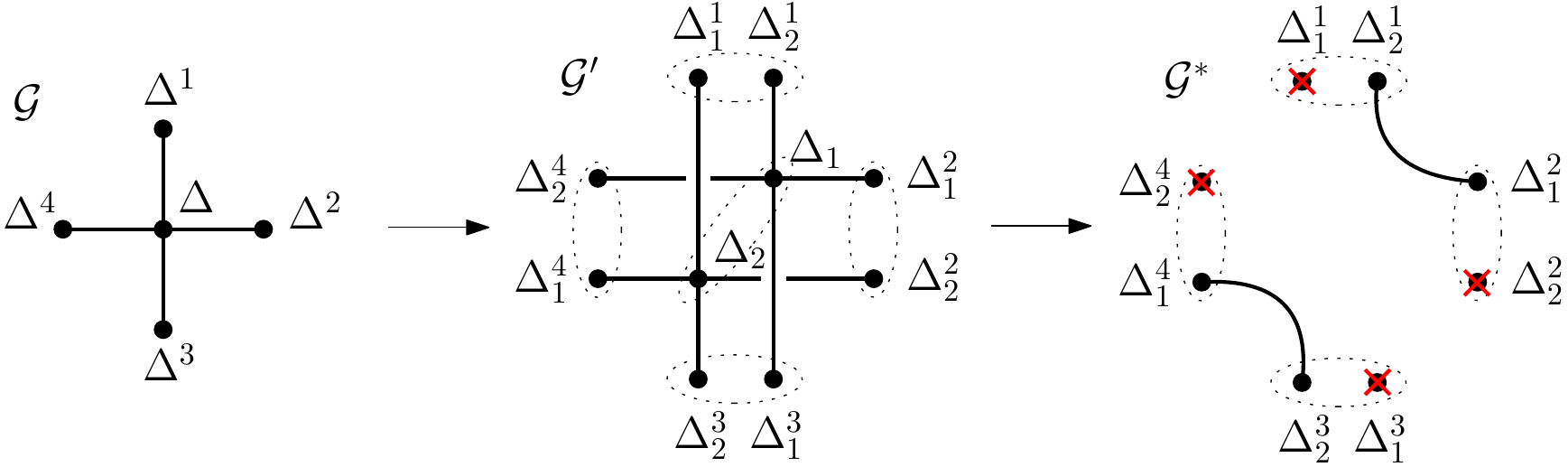}
  \caption{Local transformation of the dual graph at tetrahedron $\Delta$ from~\reffig{Collapse} when crushing iteratively at $\Delta$. $\G$ is the dual graph before crushing, $\G'$ is the dual graph after cutting, collapsing, and then flattening only the 3 and 4-sided footballs from $\Delta$ ($\Delta_1$ and $\Delta_2$ stand for the two triangular purses), and $\G^*$ is the dual graph after flattening the triangular purses. Note that some of the nodes $\Delta_i^j$ in $\G'$ may have already been removed when flattening adjacent footballs, and some of the bigon faces of triangular purses may already be collapsed. This does not change the analysis as it only removes nodes and arcs from the dual graph. 
  $\G^*$ is obtained from $\G$ by lifting $\Delta^1\Delta\Delta^2$ and $\Delta^3\Delta\Delta^4$, then removing the node $\Delta$. Because their corresponding cells have bigonal faces, and hence cannot be tetrahedra, the crossed out nodes on $\G^*$ will be removed from the graph when flattening adjacent cells. The immersion of $\G^*$ into $\G$ is obtained by mapping the (non-removed) nodes $\Delta_i^j$ in $\G^*$ to $\Delta^j$ in $\G$.}
  \label{Fig:Lifting}
\end{figure} 

\begin{corollary}\label{Cor:carving-width}
Crushing does not increase carving-width. 
\end{corollary}

\begin{proof}
Immersion does not increase carving-width, \reflem{Immersion}. 
\end{proof}

\begin{corollary}\label{Cor:CrushingTreewidth}
Crushing an arbitrary finite number of normal surfaces in a triangulation increases the treewidth by at most a multiplicative factor of six. 
\end{corollary}

\begin{proof}
This follows from \refcor{carving-width} and \refthm{boundcngtw}, using the fact that a graph coming from the dual of a 3-manifold triangulation has all nodes of degree at most four. 
\end{proof}

\begin{figure}
\centering
\includegraphics[width=10cm]{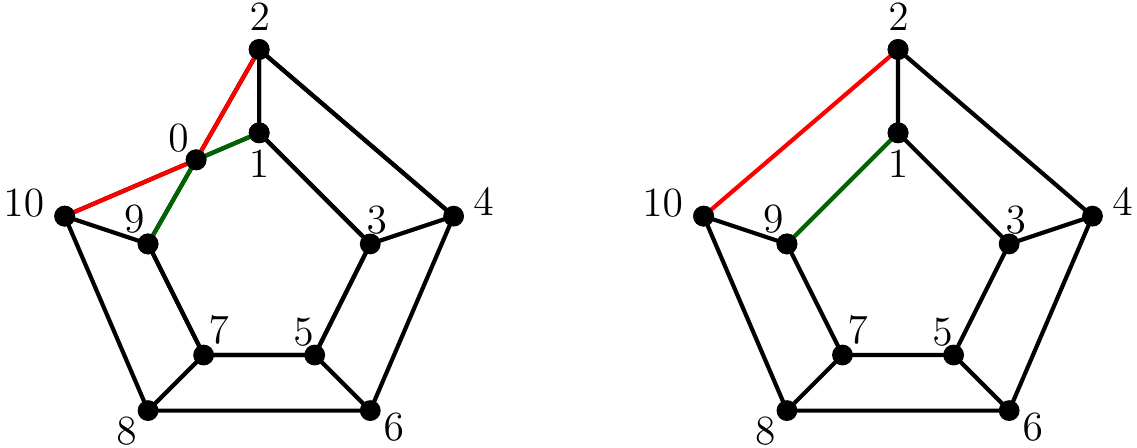}
\caption{Liftings of 2-0-10 and 1-0-9 increase the treewidth of the graph.}
\label{Fig:twincr}
\end{figure}

\begin{remark}
Note that in general the two liftings pictured in \reffig{Lifting} may increase the treewidth of a graph. For example, consider the graphs in \reffig{twincr}, where the graph of the pentagonal prism (on the right) is obtained from liftings 2-0-10 and 1-0-9 in the graph on the left. The latter has treewidth $3$, as it admits $K_4$ as a minor ($\tw \geq 3$), and a path decomposition $(B_i)_{i = 1 \ldots 8}$, $B_i = \{0,i,i+1,i+2\}$, of width $3$. The former is a well-known obstruction to treewidth $3$, and has treewidth $4$~\cite{ARNBORG19901,doi:10.1002/net.3230200304}.
\end{remark}

In their seminal work on 0-efficient triangulations, Jaco and Rubinstein proved that a minimal triangulation (i.e.\ with a minimal number of tetrahedra) of a manifold is 0-efficient. In the same spirit, we deduce the following for triangulation width:

\begin{corollary}
Any closed, orientable, irreducible 3-manifold $M$, not $S^3$, $S^2 \times S^1$, $\RP^3$ or $L(3,1)$, admits a 0-efficient triangulation of optimal carving-width $\cng(M)$.

Any compact, orientable, irreducible, $\partial$-irreducible 3-manifold $M$, not the 3-cell, admits a 0-efficient triangulation of optimal carving-width $\cng(M)$.
\end{corollary}

\begin{proof}
Let $\calT$ be a triangulation of $M$ of carving-width $\cng(M)$. By~\refthm{JRCrushing}, one can crush normal spheres and disks in $\calT$ to get a 0-efficient triangulation $\calT^*$ of carving-width at most $\cng(M)$ by~\refcor{carving-width}.
\end{proof}

To conclude this section, we prove the following simple property of carving-width.

\begin{lemma}\label{Lem:cngconnecting}
Let $\G$ and $\G'$ be two graphs, and let $\G \# \G'$ be obtained by adding $m \geq 1$ arcs between nodes of $\G$ and nodes of $\G'$, not counting multiplicities. Then
\[
\cng(\G \# \G') \leq \max \{ \cng(\G) + m-1, \cng(\G') + m-1, \text{\emph{max degree in}} \ \G \# \G' \}.
\]
If the $m$ arcs are incident to a single node $u$ in $\G$, then   
\[
\tw(\G \# \G') \leq \max \{ \tw(\G), \tw(\G') + 1 \}.
\] 
\end{lemma} 

\begin{proof}
Pick two optimal tree embeddings $\pi\from\G\to T$ and $\pi'\from\G'\to T'$, with arcs $a\in T$ and $a'\in T'$ realising the congestion $\cng(\pi)$ and $\cng(\pi')$, respectively. Without loss of generality, suppose $\cng(\pi)\geq\cng(\pi')$, hence at least as many paths run through $a$ as $a'$. Let $u$ in $\G$ and $v$ in $\G'$ be two nodes that are adjacent in $\G \# \G'$. Subdivide the only arcs incident to the leaves of $\pi(u)$ and $\pi'(v)$ in the tree embeddings, and connect the two new nodes by an arc. This leads to a tree embedding $\Pi(\G \# \G')$ of $\G\#\G'$.

Note that $m$ paths run through the arc connecting the trees $T$ and $T'$. The congestion $\cng(\Pi)$ will be largest if as many of those paths as possible also run through $a$. If $a$ in $T$ does not connect the leaf $\pi(u)$, then at most $m-1$ new paths run through $a$, because the path from $\Pi(u)$ to $\Pi(v)$ only runs through new arcs. 

If $a$ in $T$ is the arc connecting the leaf $\pi(u)$, then $\cng(\G)$ paths run to $u$, hence $u$ is $\cng(\G)$-valent. The arc $a$ is subdivided to form the new tree, and $\cng(\G)$ paths from the tree $T$ will continue to run over the two new arcs obtained by subdividing $a$. If one of the $m$ new arcs between $\G$ and $\G'$ does not have an endpoint on $u$, then the corresponding path will run over the subarc of $a$ that does not meet the leaf of $\Pi(u)$, whereas the arc from $\Pi(u)$ to $\Pi(v)$ will not meet this arc, and thus the congestion is at most $\cng(\G)+m-1$. However, if all the $m$ new arcs between $\G$ and $\G'$ run from $u$ to nodes of $\G'$, then all $m$ new paths must also run over the arc connecting the leaf of $\Pi(u)$. Thus the congestion will be at most $\cng(\G)+m$, which is the degree of $u$ in $\G\#\G'$. 

For treewidth, pick two optimal tree decompositions $T$ and $T'$ for $\G$ and $\G'$ respectively, and let $u$ in $\G$ be the node to which all new arcs are incident. Let $(u,v)$, with $v$ in $\G'$, be a new arc. Let $B_u$ be a bag of $T$ containing $u$, and $B_v$ be a bag of $T'$ containing $v$. Connecting $B_u$ and $B_v$ with an arc, and adding node $u$ to all bags in $T'$, leads to a tree decomposition of $\G \# \G'$, and the result follows.
\end{proof}

\section{Triangulating thick hyperbolic manifolds} 
\label{Sec:geometry}

In this section, we consider a finite volume compact hyperbolic $3$-manifold $M$ with boundary and bounded injectivity radius. We show such a manifold admits a triangulation with a bounded number of tetrahedra, where the bound is linear in volume. This result is not new; its proof is outlined in Thurston's notes \cite{Thurston:notes}, and proved carefully elsewhere, for example by Kobayashi and Reick \cite{KobayashiRieck}.
We step through highlights of the proof here for completeness, and also to discuss the algorithmic nature of the argument, in order to actually compute a triangulation. The algorithm will be summarised in \refsec{Algorithm}.

\subsection{Hyperbolic manifolds}\label{Sec:HyperbolicReview}

Here we review definitions and results in hyperbolic geometry that are most important to our results. For further information on hyperbolic 3-manifolds, see for example~\cite{BenedettiPetronioHyperbolicGeom}.

\begin{definition}\label{Def:ThickPart}
  Let $\mu>0$, and let $M$ be a hyperbolic 3-manifold. The \emph{$\mu$-thick part of $M$}, denoted $\Mmu$, consists of all points $x\in M$ such that any geodesic based at $x$ has length at least $\mu$. Equivalently, $\Mmu$ is the set of points in $M$ with injectivity radius at least $\mu/2$.

  The complement of the $\mu$-thick part is the $\mu$-thin part.
\end{definition}

Recall that by the Margulis lemma, there exists a universal constant $\epsilon_3$ such that for any finite volume hyperbolic 3-manifold $M$, and any $\mu\leq\epsilon_3$, the $\mu$-thin part of $M$ consists only of tubes about geodesics and cusps~\cite{KazhdanMargulis}. In the discussion below, we will always assume that $0<\mu\leq \epsilon_3$.
Such a $\mu$ is said to be a 3-dimensional \emph{Margulis constant}. 

Let $B_D(x,r)$ denote the open ball of centre $x$ and radius $r > 0$ in $D$. Recall that the volume of a hyperbolic $3$-ball of radius $r > 0$ is given by:
\[
  \vol(B_{\HH^3}(r)) = \pi (\sinh 2r - 2r);
\]
see, for example~\cite{FenchelHyperbolicGeom}.

\subsection{Triangulating thick parts}\label{Sec:Meshing} 
  
In this section, for a fixed 3-dimensional Margulis constant $\mu$, we recall the argument of \cite{KobayashiRieck} to show that a small neighbourhood of the $\mu$-thick part of a hyperbolic 3-manifold $M$ can be triangulated with $O(\vol(M))$ tetrahedra.

We start by setting notation. For $\mu>0$ a Margulis constant and any $d>0$, denote the metric $d$-neighbourhood of $\Mmu$ by $X:=N_d(\Mmu)$. In \cite[Proposition~1.2]{KobayashiRieck}, it is shown that there exists $R:=R(\mu,d)$ such that for any complete finite volume hyperbolic 3-manifold $M$, and any $x\in X$, the injectivity radius of $x$ is at least $R$, and $X$ is obtained from $M$ by drilling out short geodesics and truncating cusps. Let $D=\min\{R,d\}$.

\begin{definition}\label{Def:Net}
Let $(X,\dist_X)$ be a metric space. For $\varepsilon > 0$, a set of points $P \subset X$ is $\varepsilon$-dense in $X$ if, for any $x \in X$, there is a point $p \in P$ such that $\dist_X(x,p) < \varepsilon$. 
For $1 \geq \delta > 0$, the set $P$ is $\delta \varepsilon$-separated if any two distinct points $p,q \in P$ satisfy $\dist_X(p,q) \geq \delta \varepsilon$. 
We call $P$ a $(\delta,\varepsilon)$-net if it is $\varepsilon$-dense and $\delta \varepsilon$-separated. Note that any $(\delta,\varepsilon)$-net is also a $(\delta',\varepsilon)$-net for any $\delta' \leq \delta$.
\end{definition}

\begin{lemma}\label{Lem:mesh}
Let $\mu>0$ be a 3-dimensional Margulis constant and $d>0$. Let $M$ be a hyperbolic manifold of finite volume $\vol(M)$, with $d$-neighbourhood of the thick part $\Mmu$ denoted by $X$. Then $X$ admits a $(\delta,\varepsilon)$-net of size
\[
n \leq \frac{\vol(M)}{\pi (\sinh \varepsilon - \varepsilon)} \leq \frac{6}{\pi} \frac{\vol(M)}{\varepsilon^3}
\]
for any $\varepsilon \leq \mu$ and any $\delta \leq 1$.
\end{lemma}

\begin{proof}
This is the standard iterative construction of nets. Fix an arbitrary $\varepsilon \leq \mu$. 
Set $P$ to be the empty set.  While there exists a point $x$ in the set
\[ X - \bigcup_{p \in P} B_{X}(p, \varepsilon), \]
set $P$ to be $P \cup \{x\}$. 
At any time of the procedure, the union of balls of radius $\varepsilon/2$ centred on the points of $P$ are disjoint and embedded in $M$. Consequently, 
\[
|P| \times \vol(B_{\HH^3}(\varepsilon/2)) \leq \vol(M)
\]
Because $M$ has finite volume, the procedure terminates, and $P$ is a $(1,\varepsilon)$-net for $X$ by construction. 
\end{proof}

Recall that $D=\min\{R,d\}$. Kobayashi and Rieck take a maximal $D$-separated set for $X$, but it suffices for their argument to let $\{x_1, \dots, x_N\}$ be a $(1,D)$-net for $X$. Now let $\{V_1, \dots, V_N\}$ be the \emph{Voronoi cells} in $M$ corresponding to $\{x_1, \dots, x_N\}$, namely the sets
\[ V_i = \{p\in M \mid \dist(p,x_i) \leq \dist(p,x_j) \mbox{ for } j = 1, \dots, N\}. \]

Kobayashi and Rieck show that the components of $V_i\cap X$ consist of handlebodies with universally bounded genus, with boundaries consisting of geodesic faces meeting in geodesic edges and vertices, and that the number of such faces and edges (and vertices) is universally bounded independent of $M$. 

After possibly perturbing the points $\{x_1, \dots, x_N\}$ slightly, they give an algorithm that builds, for each component $V_{i,j}$ of $V_i\cap X$, a 2-complex $K_{i,j}$. The complex $K_{i,j}$ has totally geodesic faces, a universally bounded number of faces and edges, and it cuts $V_{i,j}$ into a single ball $B_{i,j}$.

By subdividing remaining faces into triangles, and then coning to the centre of the ball $B_{i,j}$, we obtain a triangulation of $V_{i,j}$ such that by construction, triangulations of distinct $V_{i,j}$ agree on their intersections. 

The above gives the following, which is \cite[Proposition~1.4]{KobayashiRieck}. 

\begin{proposition}\label{Prop:VoronoiTriangulations}
  Let $\mu$ be a 3-dimensional Margulis constant, and fix $d>0$. For any complete finite volume hyperbolic 3-manifold $M$, let $X$ denote the metric $d$-neighbourhood of $\Mmu$. Then there exists a constant $C=C(\mu,d)$ so that the following holds.
  \begin{enumerate}
  \item $M$ is decomposed into $N\leq C \vol(M)$ Voronoi cells $\{V_1, \dots, V_N\}$.
  \item $V_i\cap X$ is triangulated using at most $C$ tetrahedra for all $i=1, \dots, N$.
  \item For any $i,j\in\{1, \dots, N\}$, the triangulations in (2) coincide on $(V_i\cap X) \cap (V_j\cap X)$.
  \end{enumerate}
\end{proposition}

We now obtain the following consequence, which is \cite[Theorem~1.1]{KobayashiRieck}.

\begin{theorem}\label{Thm:MinVolTet}
Let $\mu$ be a 3-dimensional Margulis constant, and fix $d>0$. Then there exists a constant $v(\mu,d)>0$ such that for $M$ any closed hyperbolic 3-manifold with volume $\vol(M)$, the metric $d$-neighbourhood of the $\mu$-thick part $\Mmu$ admits a triangulation $\calT_B$ with number of tetrahedra at most
\[
   \frac{\vol(M)}{v(\mu,d)} = O(\vol(M)).
\]
\end{theorem}

\begin{proof}
Apply \refprop{VoronoiTriangulations} to $M$, decomposing $X$ into at most $C$ cells with at most $C$ tetrahedra each. Because the triangulations match where the cells overlap, this gives a triangulation of $N_d(\Mmu)$ with at most $C^2\vol(M)$ tetrahedra. Set $v(\mu,d)=1/C^2$. 
\end{proof}

Naturally, this triangulation has carving-width (and treewidth) at most $O(\vol(M))$.

\section{Triangulations of closed manifolds}
\label{Sec:combinatorics}

For $\mu$ a Margulis constant, and fixed $d>0$, let $\Mmu$ denote the thick part of $M$, and let $X:=N_d(\Mmu)$ denote its $d$-neighbourhood. Then $X$ is a compact manifold with (possible) torus boundary components; see \cite[Proposition~2.1]{KobayashiRieck}. Let $\calT_B$ be the triangulation of $X$ constructed in \refthm{MinVolTet}. This triangulation contains at most $\vol(M)/v(\mu,d)$ tetrahedra.

\begin{lemma}\label{Lem:0Efficient}
There exists a triangulation $\calT_{JR}$ of $X$ that has exactly one vertex in each boundary component of $\Mmu$, no other vertices, and contains at most $\vol(M)/v(\mu,d)$ tetrahedra. 
\end{lemma}

\begin{proof}
For the proof, we wish to obtain a $0$-efficient triangulation, as in work of Jaco and Rubinstein~\cite{JacoRubinstein:0Eff}, repeated in \refthm{JRCrushing}. To obtain a $0$-efficient triangulation from $\calT_B$, apply the crushing procedure. Crushing does not increase the number of tetrahedra. However, it may affect the topology of the underlying manifold, but only in well-understood ways that are listed in \cite[Theorem~2]{Burton:Crushing}:
\begin{itemize}
\item It may undo connect sums, 
\item cut open along properly embedded discs, 
\item fill a boundary sphere with a 2-ball, or 
\item delete a 3-ball, 3-sphere, $\RR P^3$, $L(3,1)$ or $S^2\times S^1$ component.
\end{itemize}

Since we start with $M$ hyperbolic, the interior of the manifold $\Mmu$ is also hyperbolic. Hence there are no connect sums, no boundary spheres, and no non-hyperbolic components. There are also no essential discs; thus if we cut along a properly embedded disc, the disc will cut off a ball, which we may ignore. 

Thus repeatedly applying the crushing move gives a $0$-efficient triangulation, which has the properties required by the lemma. 
\end{proof}

\begin{lemma}\label{Lem:FillingAndTreewidth}
  If $\calT$ is a triangulation of a 3-manifold $M$ with a torus boundary component $S$ such that $S$ inherits a one-vertex triangulation from $\calT$, then any Dehn filling of $M$ along $S$ can be given a triangulation with carving-width at most 
  \[
  	\max \{ \cng(\calT) + 1, 4 \}.
  \]
\end{lemma}

\begin{proof}
First, note that the only one-vertex triangulation of the torus $S$ has two triangles, so the boundary component of $\calT$ corresponding to $S$ has two triangles, which are incident to at most two distinct tetrahedra. 

The Dehn filling can be obtained by attaching a layered solid torus, as in Example~\ref{Example:LST}, with gluing graph a simple daisy chain, with one loop arc with both endpoints on the node corresponding to the tetrahedron in the centre, and pairs of parallel arcs between nodes corresponding to the layers of tetrahedra. The layering is determined by the slope of the Dehn filling. After a finite number of steps, the meridian of the layered solid torus will correspond to the desired slope; see \cite{JacoRubinstein:LayeredTriang}. At this stage, the final pair of faces is attached to the tetrahedra in $M$ that form the torus boundary.

Because the triangulation of $M$ and the one of the layered solid torus are glued along two faces of at most two distinct tetrahedra, 
the carving-width after gluing is at most the maximum of the carving-width between the two triangulations plus 1, or the maximal degree in the dual graph after filling, which is 4, by \reflem{cngconnecting}. Because a daisy chain graph has carving-width at most two, by \reflem{DaisyChain}, the result follows.
\end{proof}

\begin{theorem}\label{Thm:TWandVol}
  Let $\mu$ be a 3-dimensional Margulis constant and fix $d>0$.
  Let $M$ be a closed hyperbolic 3-manifold with volume $\vol(M)$.
  Then the carving-width of $M$ is bounded above by $6 \cdot \vol(M)/v(\mu,d) = O(\vol(M))$. 
\end{theorem}

\begin{proof}
Start with the triangulation $\calT_{JR}$ of $X:= N_d(\Mmu)$. The triangulation $\calT_{JR}$ is obtained in \reflem{0Efficient} by crushing, and by \refcor{carving-width} has carving-width at most the carving-width of triangulation $\calT_B$ in \refthm{MinVolTet}, which is naturally at most $4 \cdot \vol(M)/v(\mu,d) = O(\vol(M))$.

We obtain $M$ from $X$ by Dehn filling the torus boundary components. Each boundary component inherits a triangulation from $\calT_{JR}$ with exactly one vertex and two triangles, and so there are at most $2 \cdot \vol(M)/v(\mu,d)$ boundary components. Consequently, performing the Dehn fillings increases the carving-width by at most $2 \cdot \vol(M)/v(\mu,d)$, by \reflem{FillingAndTreewidth}. 
\end{proof}

By \refthm{boundcngtw}, this result also applies to treewidth.

\subsection{Algorithm and computational complexity}\label{Sec:Algorithm}

Our approach to constructing a triangulation with carving-width $O(\vol(M))$ for a closed hyperbolic 3-manifold $M$ is algorithmic. Given $M$ presented by its thick part and hyperbolic Dehn fillings, and an oracle to access its geometry, one can compute the $(1,D)$-net of \reflem{mesh} with $O(\poly(\vol(M)))$ calls to the oracle, where $\poly$ is a polynomial function. The procedure to compute Voronoi cells is polynomial in the size of the net, which is $O(\vol(M))$. 
Meshing in polynomial time, while maintaining the Voronoi cells, can be done using an intrinsic version of Delaunay refinement~\cite{edelsbrunner_2001}. 
Kobayashi and Rieck's algorithm to subdivide each cell $V_{i,j}$ into a ball by constructing a 2-complex $K_{i,j}$ requires first building a graph with desirable properties, in \cite[Lemma~4.2]{KobayashiRieck}, which is constructed by adding components of $V_{i,j}\cap\bdy X$ one at a time, and then adding and removing edges that lie in geodesic faces. The number of components of $V_{i,j}\cap \bdy X$ is universally bounded, depending on the universal bound of the genus of $V_{i,j}$, and thus for each $V_{i,j}$ this algorithm runs in constant time. One can then triangulate the balls $V_{i,j}-K_{i,j}$ in constant time, using the fact that the number of faces, edges, and vertices of each ball is universally bounded. This is done for each of the $O(\vol(M))$ cells. 

The algorithm using the crushing procedure in the proof of \reflem{0Efficient} is exponential in $\vol(M)$, the number of tetrahedra. However, the complexity becomes polynomial~\cite{Burton:Crushing} if, instead of $0$-efficient triangulations, we only require the triangulations $\calT_i$ of the decomposition described in \refthm{JRCrushing} to have exactly one vertex per boundary component if bounded, or one vertex triangulation if closed. This is sufficient for the construction of \refthm{TWandVol}.

The filling by Dehn surgery in the proof of \refthm{TWandVol} is linear time in the Dehn filling coefficients (and hence linear in the final number of tetrahedra for the output triangulation). Consequently, the procedure above constructs a triangulation with treewidth $\vol(M)$ for a closed hyperbolic 3-manifold $M$ in time:
\[
O\left( \poly(\vol(M)) \cdot \oracle + \poly(\vol(M)) + n\right)
\]
and $O(n)$ memory, where $\poly$ denotes polynomial functions, $\oracle$ is the time complexity for calling the geometry oracle, 
and $n$ is the number of tetrahedra of the output manifold. Note that, by construction, we also get a tree-decomposition of width $O(\vol(M))$.

\section{Treewidth does not bound volume}
\label{Sec:cst_tw}

In this section, we prove there exist families of manifolds with constant treewidth but unbounded volume. Our examples include both manifolds with boundary and closed manifolds.

The manifolds with boundary that we consider are the exteriors of 2-bridge knots. There are many ways to describe 2-bridge knots; see for example \cite{BurdeZieschang}. For the purpose of this paper, a 2-bridge knot $K[a_{n-1}, \ldots, a_1]$ is described by a finite collection of integers $a_1, \dots, a_{n-1}$. The diagram of the knot $K[a_{n-1}, \ldots, a_1]$ consists of $n-1$ twist regions arranged linearly, with the $i$-th region containing $|a_i|$ crossings, and the direction of the crossing determined by the sign of $a_i$, and twist regions connected as shown in \reffig{twobridgeknot}. In general, we may always assume that either all the $a_i$ are positive or all negative, and $|a_1|$ and $|a_{n-1}|$ are at least $2$. 

\begin{figure}
  \centering
  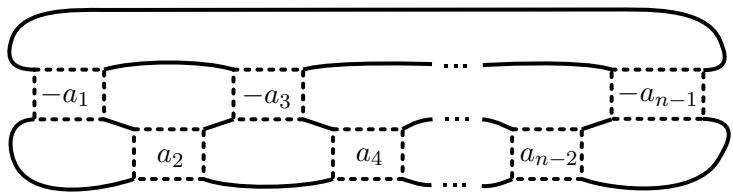
  \caption{The diagram of $K[a_{n-1}, \ldots , a_1]$, for $n$ even. The box labelled $\pm a_i$ denotes a horizontal twist region with $|a_i|$ crossings, with the sign of all crossings equal to the sign of $\pm a_i$. The crossing number is $C = |a_{n-1}| + \ldots + |a_1|$. }
  \label{Fig:twobridgeknot}
\end{figure}

\begin{proposition}\label{Prop:CuspedBddTW}
The family of hyperbolic 2-bridge knots has unbounded volume but treewidth bounded by a constant. This is true whether we take treewidth corresponding to an ideal triangulation or corresponding to a finite triangulation of the knot exterior. 
\end{proposition}

\begin{proof}
By Theorem~B.3 of \cite{GueritaudFuter:2bridge}, the volume of a 2-bridge knot $K[a_{n-1}, \dots, a_1]$ is bounded below by $2\,v_3\,n$, where $v_3=1.0149\dots$ is a constant. Thus letting $n$ approach infinity gives a sequence of 2-bridge knots with volume approaching infinity.

To show that these knots have bounded treewidth, we need to describe a triangulation of the knot complements. We use the well known triangulation of 2-bridge knot complements due to Sakuma and Weeks \cite{SakumaWeeks}; see also Gu{\'e}ritaud and Futer \cite{GueritaudFuter:2bridge}. We will review the description of ideal triangulation briefly here. More details can be found in the two previous references, or in \cite{Purcell:book}.

For ease of exposition, we will only work with examples for which $n$ is even, each $a_i>0$, and $a_1$ and $a_{n-1}$ are both at least $2$. While a similar argument works when $n$ is odd or when the values $a_i$ are all negative, we will not need it for our purposes here. 

The easiest way to describe the triangulation is to start with the diagram of $K[a_{n-1}, \dots, a_1]$ as in \reffig{twobridgeknot}, and then isotope all the odd twist regions to be vertical, as in \reffig{tangle}.

\begin{figure}
  \centering
  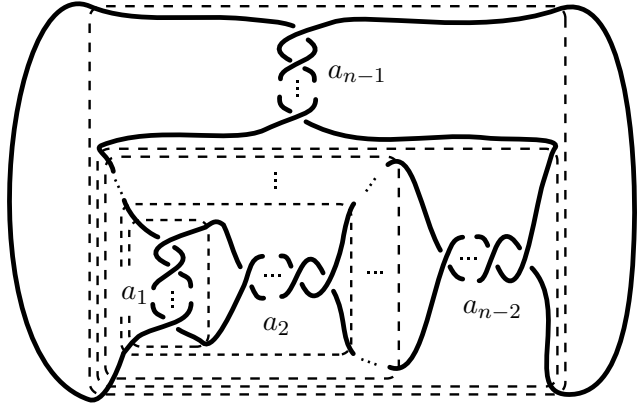
  \caption{Another diagram of a 2-bridge knot.}
  \label{Fig:tangle}
\end{figure}

With the diagram in the form of \reffig{tangle}, we may think of the crossings as nested, with the first crossing of the twist region corresponding to the first crossing of $a_1$ on the inside, and the last crossing corresponding to the last crossing of $a_{n-1}$ on the outside. Recall that there are $C = a_1+\dots +a_{n-1}$ crossings in total. The complement of the knot is built of the following:
\begin{enumerate}
\item A tangle containing the very first crossing on the inside.
\item For the $i$-th crossing, $i=2, \dots, C-1$, a block homeomorphic to $S\times I$ where $S$ is a 4-punctured sphere. The block contains a single horizontal or vertical crossing, depending on whether the $i$-th crossing is horizontal or vertical. See \reffig{4PunctSphereBlocks}. It is stacked onto the previous crossing.
\item A tangle containing the very last crossing on the outside.
\end{enumerate}

\begin{figure}
  \centering
  \includegraphics{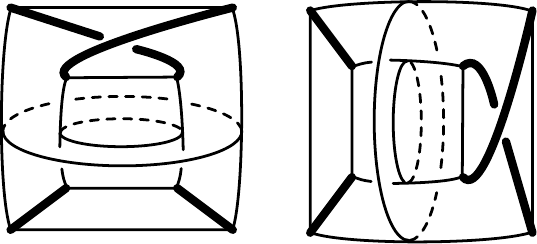}
  \caption{Vertical (left) and horizontal (right) blocks of the form $S\times I$. The 4-punctured spheres on the outside and inside correspond to $S\times\{1\}$ and $S\times\{0\}$, respectively. Figure from \cite{Purcell:book}.}
  \label{Fig:4PunctSphereBlocks}
\end{figure}

There are $2(C-2)$ tetrahedra in the decomposition, which we now describe. First, ignore the innermost and outermost tangles containing a single crossing. For the $i$-th crossing, $i=2, \dots, C-2$, there is a pair of tetrahedra lying between the $i$-th and $(i+1)$-st blocks $S\times I$. The ideal edges of these tetrahedra are edges that are either horizontal or vertical on one of the surfaces $S\times\{0\}$ and $S\times\{1\}$, for the $i$-th or $(i+1)$-st block. When we isotope all these edges to lie between the two blocks, we see the form of the two tetrahedra, as in \reffig{Pillowcase}, left. One tetrahedron, denoted $T_i^1$ lies in front of the $i$-th block, and one, denoted $T_i^2$, lies behind.

\begin{figure}
  \centering
  \includegraphics{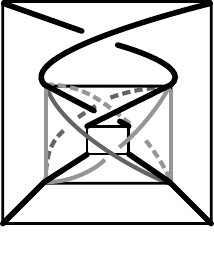}
  \hspace{.75in}
  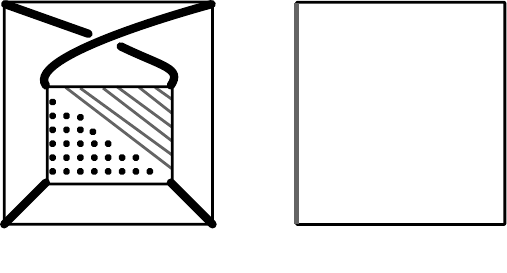
  \caption{On the far left, the edges of the tetrahedron are shown. Middle: two faces of the tetrahedron $T_i^1$ lying on the surface $S\times\{0\}\subset S\times I$ for the $(i+1)$-st block. Right: Position of those two faces when isotoped to $S\times\{1\}$ on the $(i+1)$-st block. Figure from \cite{Purcell:book}.}
  \label{Fig:Pillowcase}
\end{figure}

We need to determine how the faces of the tetrahedra glue. Note that the faces on the outside are isotopic through the outside block to triangles on the outside block. Consider the two outer faces of $T_i^1$ on $S\times\{0\}$ on the $(i+1)$-st block. These two faces are shown in \reffig{Pillowcase} for the example that the crossing is vertical. When we isotope through the $(i+1)$-st block, the triangles isotope to the triangles shown on the right of \reffig{Pillowcase}. Note that one lies in front, and one lies in back. The one in front will be glued to an inside face of $T_{i+1}^1$, and the one in the back will be glued to an inside face of $T_{i+1}^2$. Similarly, if the crossing in the $(i+1)$-st block is horizontal, triangle faces of $T_i^1$ on the outside isotope to two triangles, one on the front and one on the back of $S\times\{1\}$ in the $(i+1)$-st block. Thus one will be glued to $T_{i+1}^1$ and the other to $T_{i+1}^2$.

An identical argument, isotoping in the other direction, then implies that one of the inside faces of $T_i^1$ is glued to a face of $T_{i-1}^1$ and the other inside face is glued to $T_{i-1}^2$. 
 
To complete the dual graph of the triangulation, we must determine how inner faces of $T_2^1$ and $T_2^2$ glue when we insert the innermost tangle containing a single crossing. As described in \cite{GueritaudFuter:2bridge}, inserting this tangle glues the two faces on the block $S\times I$ containing the second crossing to each other. The two triangles in the front of $S\times\{0\}$ are glued, and the two triangles in the back of $S\times\{0\}$ are glued. However, recall that we view tetrahedra $T_2^1$ and $T_2^2$ as lying between the 2nd and 3rd blocks. As in \reffig{Pillowcase}, right, isotoping the inner faces of $T_2^1$ and $T_2^2$ through the block $S\times I$ to $S\times\{0\}$ puts the two faces of $T_2^1$ on opposite sides of $S\times\{0\}$. Thus when we attach the tangle containing the first crossing, we glue each inner face of $T_2^1$ to an inner face of $T_2^2$. A similar argument shows each outer face of $T_{C-2}^1$ is glued to an outer face of $T_{C-2}^2$.

Thus the graph dual to the triangulation of the 2-bridge knot complement has the form shown in \reffig{2BridgeDual}.

\begin{figure}
  \centering
  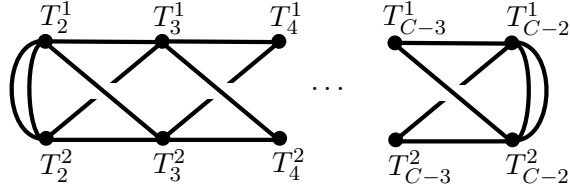
  \caption{The form of the dual graph to a 2-bridge knot triangulation.}
  \label{Fig:2BridgeDual}
\end{figure}

A tree decomposition of the graph of \reffig{2BridgeDual} is shown in \reffig{2BridgeTreeDecomp}. There are $C-4$ bags in the tree decomposition. Each bag contains exactly four nodes, namely $T_i^1$, $T_i^2$, $T_{i+1}^1$, and $T_{i+1}^2$ for $i=2, 3, \dots, C-3$. Thus the treewidth of this tree decomposition is $4-1=3$. Note it is constant, independent of the values of $a_i$ and the number of crossings $C$ and twist regions $n$.

\begin{figure}
  \centering
  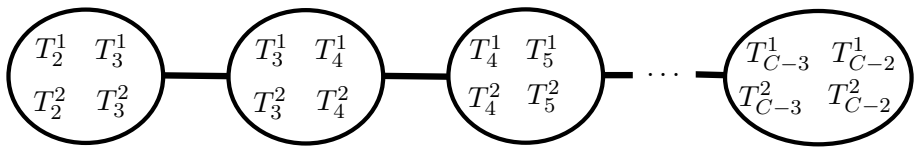
  \caption{A tree decomposition of an ideal triangulation of a 2-bridge knot.}
  \label{Fig:2BridgeTreeDecomp}
\end{figure}

Finally, the above argument holds for ideal triangulations. It might be preferable to work with finite triangulations, i.e., tetrahedra with only finite vertices and not ideal vertices. We can modify the above decomposition into a finite triangulation by first, truncating ideal vertices. This turns an ideal tetrahedron into a polyhedron with four triangular faces (from truncating) and four hexagonal faces (one for each face of the previous ideal tetrahedron). Add a finite vertex to the centre of each hexagonal face and cone to obtain six triangles. Then add a finite vertex to the centre of the polyhedron and cone to all the faces. The result is a subdivision of the ideal tetrahedron into 28 finite tetrahedra.

To obtain a tree decomposition of the finite triangulation, take the same tree decomposition as in \reffig{2BridgeTreeDecomp}. Replace $T_i^j$ with the 28 finite tetrahedra in the subdivision of $T_i^j$. This gives a tree decomposition, although it is likely not optimal. However, the size of each bag in the decomposition is $28*4 = 112$. Thus the treewidth of this decomposition is $111$, which is constant, independent of $C$ and $n$.
\end{proof}

\smallskip

\begin{remark}~\label{Rem:twknotcomplement}
It is folklore that given a link diagram of treewidth $k$ (seen as a 4-valent graph), one can construct an ideal triangulation of its link exterior with treewidth $O(k)$. More explicitly, this can be done using {\tt SnapPy}'s link complement triangulation algorithm~\cite{snappy}. Specifically, {\tt SnapPy}'s procedure first constructs a cell complex with four cells per crossing in the diagram, whose dual graph connects the four cells around a crossing in a square, and whose arcs run along the link diagram otherwise; this graph has treewidth at most $4$ times the treewidth of the diagram. The procedure concludes by contracting the cell complex along the link diagram, which induces \emph{arc contractions} in the dual graph that can only reduce the treewidth~\cite{robertson86-algorithmic}, and by adding a constant number of cells (increasing the treewidth by a constant) to get an ideal triangulation. 

However, the construction of \refprop{CuspedBddTW} above gives a more natural triangulation with a smaller treewidth. 
\end{remark}

We have found manifolds with boundary with unbounded volume and bounded treewidth. We wish to find closed examples. To do so, we will perform Dehn filling on the 2-bridge knots from above, by attaching a layered solid torus. However, we need to ensure that there is a triangulation of the manifold with only one vertex on the boundary, with bounded treewidth.

\begin{theorem}\label{Thm:BddTW}
There exists a sequence of closed hyperbolic manifolds $M_n$ with bounded treewidth and unbounded volume. 
\end{theorem}

\begin{proof}
The sequence will be obtained by Dehn filling 2-bridge knot complements. By Proposition~\ref{Prop:CuspedBddTW}, there exists a sequence of 2-bridge knots with volume approaching infinity but with treewidth bounded by a constant. 

By virtue of \refcor{carving-width} (and \refthm{boundcngtw}), applying the 0-efficiency construction of Jaco and Rubinstein~\cite{JacoRubinstein:0Eff}, any (compact) 2-bridge knot exterior admits a bounded triangulation with constant treewidth and one vertex on each boundary component. 
Now perform a very high Dehn filling on the knot. The volume decreases by a bounded amount; see for example~\cite{FKP:DFVolJP}. Hence the sequence still has volume approaching infinity.

We construct the triangulation of the Dehn filling using \reflem{FillingAndTreewidth}, which increases the treewidth by one, by \reflem{cngconnecting}. Thus the treewidth remains bounded.
\end{proof}

\section*{Acknowledgements}  
CM would like to thank Arnaud de Mesmay and Jonathan Spreer for discussions that led to the simple argument connecting treewidth of link diagrams and link exteriors in \refrem{twknotcomplement}. JP was supported in part by the Australian Research Council. 

\bibliography{biblio}
\bibliographystyle{amsplain}

\end{document}

%% file: figures/Daisy_CarvingWidth.pdf_tex
\begingroup%
  \makeatletter%
  \providecommand\color[2][]{%
    \errmessage{(Inkscape) Color is used for the text in Inkscape, but the package 'color.sty' is not loaded}%
    \renewcommand\color[2][]{}%
  }%
  \providecommand\transparent[1]{%
    \errmessage{(Inkscape) Transparency is used (non-zero) for the text in Inkscape, but the package 'transparent.sty' is not loaded}%
    \renewcommand\transparent[1]{}%
  }%
  \providecommand\rotatebox[2]{#2}%
  \ifx\svgwidth\undefined%
    \setlength{\unitlength}{332.27277374bp}%
    \ifx\svgscale\undefined%
      \relax%
    \else%
      \setlength{\unitlength}{\unitlength * \real{\svgscale}}%
    \fi%
  \else%
    \setlength{\unitlength}{\svgwidth}%
  \fi%
  \global\let\svgwidth\undefined%
  \global\let\svgscale\undefined%
  \makeatother%
  \begin{picture}(1,0.20607001)%
    \put(0,0){\includegraphics[width=\unitlength,page=1]{Daisy_CarvingWidth.pdf}}%
    \put(0.41830428,0.12749824){\color[rgb]{0,0,0}\makebox(0,0)[lb]{\smash{$a_1$}}}%
    \put(0.3223848,0.12835828){\color[rgb]{0,0,0}\makebox(0,0)[lb]{\smash{$a_2$}}}%
    \put(0.2273256,0.12792826){\color[rgb]{0,0,0}\makebox(0,0)[lb]{\smash{$a_3$}}}%
    \put(0.09098636,0.11880364){\color[rgb]{0,0,0}\makebox(0,0)[lb]{\smash{$a_{n-1}$}}}%
    \put(0.00193634,0.12620763){\color[rgb]{0,0,0}\makebox(0,0)[lb]{\smash{$a_n$}}}%
    \put(0.93145241,0.02752776){\color[rgb]{0,0,0}\makebox(0,0)[lb]{\smash{$a_n$}}}%
    \put(0.84321696,0.00709777){\color[rgb]{0,0,0}\makebox(0,0)[lb]{\smash{$a_{n-1}$}}}%
    \put(0.76063733,0.04494946){\color[rgb]{0,0,0}\makebox(0,0)[lb]{\smash{$a_3$}}}%
    \put(0.70446826,0.07796497){\color[rgb]{0,0,0}\makebox(0,0)[lb]{\smash{$a_2$}}}%
    \put(0.66618646,0.11839751){\color[rgb]{0,0,0}\makebox(0,0)[lb]{\smash{$a_1$}}}%
    \put(0,0){\includegraphics[width=\unitlength,page=2]{Daisy_CarvingWidth.pdf}}%
    \put(0.57593024,0.17170377){\color[rgb]{0,0,0}\makebox(0,0)[lb]{\smash{$\pi$}}}%
  \end{picture}%
\endgroup%

%% file: figures/Collapse_modified.pdf_tex
\begingroup%
  \makeatletter%
  \providecommand\color[2][]{%
    \errmessage{(Inkscape) Color is used for the text in Inkscape, but the package 'color.sty' is not loaded}%
    \renewcommand\color[2][]{}%
  }%
  \providecommand\transparent[1]{%
    \errmessage{(Inkscape) Transparency is used (non-zero) for the text in Inkscape, but the package 'transparent.sty' is not loaded}%
    \renewcommand\transparent[1]{}%
  }%
  \providecommand\rotatebox[2]{#2}%
  \ifx\svgwidth\undefined%
    \setlength{\unitlength}{372.51389271bp}%
    \ifx\svgscale\undefined%
      \relax%
    \else%
      \setlength{\unitlength}{\unitlength * \real{\svgscale}}%
    \fi%
  \else%
    \setlength{\unitlength}{\svgwidth}%
  \fi%
  \global\let\svgwidth\undefined%
  \global\let\svgscale\undefined%
  \makeatother%
  \begin{picture}(1,0.44463263)%
    \put(0,0){\includegraphics[width=\unitlength,page=1]{Collapse_modified.pdf}}%
    \put(0.03966048,0.32682317){\color[rgb]{0,0,0}\makebox(0,0)[lb]{\smash{$\Delta$}}}%
    \put(0.39035081,0.42131469){\color[rgb]{0,0,0}\makebox(0,0)[lb]{\smash{collapse}}}%
    \put(0.73032562,0.42423705){\color[rgb]{0,0,0}\makebox(0,0)[lb]{\smash{flatten}}}%
    \put(0.39716978,0.35117652){\color[rgb]{0,0,0}\makebox(0,0)[lb]{\smash{triangular purse}}}%
    \put(0.46146301,0.20310732){\color[rgb]{0,0,0}\makebox(0,0)[lb]{\smash{3-sided football}}}%
    \put(0.46438546,0.11738308){\color[rgb]{0,0,0}\makebox(0,0)[lb]{\smash{4-sided football}}}%
    \put(0.40885951,0.03945187){\color[rgb]{0,0,0}\makebox(0,0)[lb]{\smash{triangular purse}}}%
    \put(0.62706683,0.40183184){\color[rgb]{0,0,0}\makebox(0,0)[lb]{\smash{$\Delta_1$}}}%
    \put(0.6358341,0.00633109){\color[rgb]{0,0,0}\makebox(0,0)[lb]{\smash{$\Delta_2$}}}%
    \put(0.77708434,0.22259016){\color[rgb]{0,0,0}\makebox(0,0)[lb]{\smash{edge}}}%
    \put(0.89787766,0.28298682){\color[rgb]{0,0,0}\makebox(0,0)[lb]{\smash{triangle}}}%
  \end{picture}%
\endgroup%

%% file: figures/2BridgeDiagram.pdf_tex
\begingroup%
  \makeatletter%
  \providecommand\color[2][]{%
    \errmessage{(Inkscape) Color is used for the text in Inkscape, but the package 'color.sty' is not loaded}%
    \renewcommand\color[2][]{}%
  }%
  \providecommand\transparent[1]{%
    \errmessage{(Inkscape) Transparency is used (non-zero) for the text in Inkscape, but the package 'transparent.sty' is not loaded}%
    \renewcommand\transparent[1]{}%
  }%
  \providecommand\rotatebox[2]{#2}%
  \ifx\svgwidth\undefined%
    \setlength{\unitlength}{272.60944176bp}%
    \ifx\svgscale\undefined%
      \relax%
    \else%
      \setlength{\unitlength}{\unitlength * \real{\svgscale}}%
    \fi%
  \else%
    \setlength{\unitlength}{\svgwidth}%
  \fi%
  \global\let\svgwidth\undefined%
  \global\let\svgscale\undefined%
  \makeatother%
  \begin{picture}(1,0.25466212)%
    \put(0,0){\includegraphics[width=\unitlength,page=1]{2BridgeDiagram.pdf}}%
    \put(0.04532513,0.12422765){\color[rgb]{0,0,0}\makebox(0,0)[lb]{\smash{$-a_1$}}}%
    \put(0,0){\includegraphics[width=\unitlength,page=2]{2BridgeDiagram.pdf}}%
    \put(0.20665025,0.04029607){\color[rgb]{0,0,0}\makebox(0,0)[lb]{\smash{$a_2$}}}%
    \put(0.3240611,0.12199684){\color[rgb]{0,0,0}\makebox(0,0)[lb]{\smash{$-a_3$}}}%
    \put(0.48162942,0.04210071){\color[rgb]{0,0,0}\makebox(0,0)[lb]{\smash{$a_4$}}}%
    \put(0.70853259,0.04238028){\color[rgb]{0,0,0}\makebox(0,0)[lb]{\smash{$a_{n-2}$}}}%
    \put(0.84317337,0.12477663){\color[rgb]{0,0,0}\makebox(0,0)[lb]{\smash{$-a_{n-1}$}}}%
    \put(0,0){\includegraphics[width=\unitlength,page=3]{2BridgeDiagram.pdf}}%
    \put(0.36838382,0.11991263){\color[rgb]{0,0,0}\makebox(0,0)[lb]{\smash{}}}%
  \end{picture}%
\endgroup%

%% file: figures/PosTangleDiagram.pdf_tex
\begingroup%
  \makeatletter%
  \providecommand\color[2][]{%
    \errmessage{(Inkscape) Color is used for the text in Inkscape, but the package 'color.sty' is not loaded}%
    \renewcommand\color[2][]{}%
  }%
  \providecommand\transparent[1]{%
    \errmessage{(Inkscape) Transparency is used (non-zero) for the text in Inkscape, but the package 'transparent.sty' is not loaded}%
    \renewcommand\transparent[1]{}%
  }%
  \providecommand\rotatebox[2]{#2}%
  \ifx\svgwidth\undefined%
    \setlength{\unitlength}{237.32792664bp}%
    \ifx\svgscale\undefined%
      \relax%
    \else%
      \setlength{\unitlength}{\unitlength * \real{\svgscale}}%
    \fi%
  \else%
    \setlength{\unitlength}{\svgwidth}%
  \fi%
  \global\let\svgwidth\undefined%
  \global\let\svgscale\undefined%
  \makeatother%
  \begin{picture}(1,0.63729677)%
    \put(0,0){\includegraphics[width=\unitlength,page=1]{PosTangleDiagram.pdf}}%
    \put(0.18130635,0.16406641){\color[rgb]{0,0,0}\makebox(0,0)[lb]{\smash{$a_1$}}}%
    \put(0.50920708,0.51623735){\color[rgb]{0,0,0}\makebox(0,0)[lb]{\smash{$a_{n-1}$}}}%
    \put(0.40580584,0.11451473){\color[rgb]{0,0,0}\makebox(0,0)[lb]{\smash{$a_2$}}}%
    \put(0.72384681,0.14181871){\color[rgb]{0,0,0}\makebox(0,0)[lb]{\smash{$a_{n-2}$}}}%
  \end{picture}%
\endgroup%

%% file: figures/TetrFaces.pdf_tex
\begingroup%
  \makeatletter%
  \providecommand\color[2][]{%
    \errmessage{(Inkscape) Color is used for the text in Inkscape, but the package 'color.sty' is not loaded}%
    \renewcommand\color[2][]{}%
  }%
  \providecommand\transparent[1]{%
    \errmessage{(Inkscape) Transparency is used (non-zero) for the text in Inkscape, but the package 'transparent.sty' is not loaded}%
    \renewcommand\transparent[1]{}%
  }%
  \providecommand\rotatebox[2]{#2}%
  \ifx\svgwidth\undefined%
    \setlength{\unitlength}{146.00548553bp}%
    \ifx\svgscale\undefined%
      \relax%
    \else%
      \setlength{\unitlength}{\unitlength * \real{\svgscale}}%
    \fi%
  \else%
    \setlength{\unitlength}{\svgwidth}%
  \fi%
  \global\let\svgwidth\undefined%
  \global\let\svgscale\undefined%
  \makeatother%
  \begin{picture}(1,0.5374865)%
    \put(0,0){\includegraphics[width=\unitlength,page=1]{TetrFaces.pdf}}%
    \put(0.66686454,0.01449338){\color[rgb]{0,0,0}\makebox(0,0)[lb]{\smash{$S\times\{1\}$}}}%
    \put(0.11971145,0.01310953){\color[rgb]{0,0,0}\makebox(0,0)[lb]{\smash{$S\times\{0\}$}}}%
    \put(0,0){\includegraphics[width=\unitlength,page=2]{TetrFaces.pdf}}%
  \end{picture}%
\endgroup%

%% file: figures/2BridgeDualGraph.pdf_tex
\begingroup%
  \makeatletter%
  \providecommand\color[2][]{%
    \errmessage{(Inkscape) Color is used for the text in Inkscape, but the package 'color.sty' is not loaded}%
    \renewcommand\color[2][]{}%
  }%
  \providecommand\transparent[1]{%
    \errmessage{(Inkscape) Transparency is used (non-zero) for the text in Inkscape, but the package 'transparent.sty' is not loaded}%
    \renewcommand\transparent[1]{}%
  }%
  \providecommand\rotatebox[2]{#2}%
  \ifx\svgwidth\undefined%
    \setlength{\unitlength}{212.15027046bp}%
    \ifx\svgscale\undefined%
      \relax%
    \else%
      \setlength{\unitlength}{\unitlength * \real{\svgscale}}%
    \fi%
  \else%
    \setlength{\unitlength}{\svgwidth}%
  \fi%
  \global\let\svgwidth\undefined%
  \global\let\svgscale\undefined%
  \makeatother%
  \begin{picture}(1,0.31060563)%
    \put(0,0){\includegraphics[width=\unitlength,page=1]{2BridgeDualGraph.pdf}}%
    \put(0.05427861,0.2743327){\color[rgb]{0,0,0}\makebox(0,0)[lb]{\smash{$T_2^1$}}}%
    \put(0,0){\includegraphics[width=\unitlength,page=2]{2BridgeDualGraph.pdf}}%
    \put(0.05427861,0.01036863){\color[rgb]{0,0,0}\makebox(0,0)[lb]{\smash{$T_2^2$}}}%
    \put(0,0){\includegraphics[width=\unitlength,page=3]{2BridgeDualGraph.pdf}}%
    \put(0.53506989,0.15985845){\color[rgb]{0,0,0}\makebox(0,0)[lb]{\smash{$\dots$}}}%
    \put(0.254945,0.01171551){\color[rgb]{0,0,0}\makebox(0,0)[lb]{\smash{$T_3^2$}}}%
    \put(0.254945,0.27298558){\color[rgb]{0,0,0}\makebox(0,0)[lb]{\smash{$T_3^1$}}}%
    \put(0.46638539,0.0130624){\color[rgb]{0,0,0}\makebox(0,0)[lb]{\smash{$T_4^2$}}}%
    \put(0.46234513,0.27298581){\color[rgb]{0,0,0}\makebox(0,0)[lb]{\smash{$T_4^1$}}}%
    \put(0.66301155,0.2743327){\color[rgb]{0,0,0}\makebox(0,0)[lb]{\smash{$T_{C-3}^1$}}}%
    \put(0.67243883,0.0090222){\color[rgb]{0,0,0}\makebox(0,0)[lb]{\smash{$T_{C-3}^2$}}}%
    \put(0.88253254,0.27298581){\color[rgb]{0,0,0}\makebox(0,0)[lb]{\smash{$T_{C-2}^1$}}}%
    \put(0.88118571,0.01036863){\color[rgb]{0,0,0}\makebox(0,0)[lb]{\smash{$T_{C-2}^2$}}}%
  \end{picture}%
\endgroup%

%% file: figures/2BridgeTreeDecomp.pdf_tex
\begingroup%
  \makeatletter%
  \providecommand\color[2][]{%
    \errmessage{(Inkscape) Color is used for the text in Inkscape, but the package 'color.sty' is not loaded}%
    \renewcommand\color[2][]{}%
  }%
  \providecommand\transparent[1]{%
    \errmessage{(Inkscape) Transparency is used (non-zero) for the text in Inkscape, but the package 'transparent.sty' is not loaded}%
    \renewcommand\transparent[1]{}%
  }%
  \providecommand\rotatebox[2]{#2}%
  \ifx\svgwidth\undefined%
    \setlength{\unitlength}{341.902565bp}%
    \ifx\svgscale\undefined%
      \relax%
    \else%
      \setlength{\unitlength}{\unitlength * \real{\svgscale}}%
    \fi%
  \else%
    \setlength{\unitlength}{\svgwidth}%
  \fi%
  \global\let\svgwidth\undefined%
  \global\let\svgscale\undefined%
  \makeatother%
  \begin{picture}(1,0.15328377)%
    \put(0.27407997,0.09687031){\color[rgb]{0,0,0}\makebox(0,0)[lb]{\smash{$T_3^1$}}}%
    \put(0.27240865,0.03837469){\color[rgb]{0,0,0}\makebox(0,0)[lb]{\smash{$T_3^2$}}}%
    \put(0.34176843,0.04004589){\color[rgb]{0,0,0}\makebox(0,0)[lb]{\smash{$T_4^2$}}}%
    \put(0.3400971,0.09603486){\color[rgb]{0,0,0}\makebox(0,0)[lb]{\smash{$T_4^1$}}}%
    \put(0.57408187,0.04673125){\color[rgb]{0,0,0}\makebox(0,0)[lb]{\smash{$T_5^2$}}}%
    \put(0.57241053,0.09603486){\color[rgb]{0,0,0}\makebox(0,0)[lb]{\smash{$T_5^1$}}}%
    \put(0.81536017,0.09328314){\color[rgb]{0,0,0}\makebox(0,0)[lb]{\smash{$T_{C-3}^1$}}}%
    \put(0.80723101,0.04506005){\color[rgb]{0,0,0}\makebox(0,0)[lb]{\smash{$T_{C-3}^2$}}}%
    \put(0.9098246,0.09364674){\color[rgb]{0,0,0}\makebox(0,0)[lb]{\smash{$T_{C-2}^1$}}}%
    \put(0.90534206,0.04707715){\color[rgb]{0,0,0}\makebox(0,0)[lb]{\smash{$T_{C-2}^2$}}}%
    \put(0,0){\includegraphics[width=\unitlength,page=1]{2BridgeTreeDecomp.pdf}}%
    \put(0.5089004,0.04338914){\color[rgb]{0,0,0}\makebox(0,0)[lb]{\smash{$T_4^2$}}}%
    \put(0.51057172,0.09603486){\color[rgb]{0,0,0}\makebox(0,0)[lb]{\smash{$T_4^1$}}}%
    \put(0,0){\includegraphics[width=\unitlength,page=2]{2BridgeTreeDecomp.pdf}}%
    \put(0.70544756,0.07781735){\color[rgb]{0,0,0}\makebox(0,0)[lb]{\smash{$\dots$}}}%
    \put(0,0){\includegraphics[width=\unitlength,page=3]{2BridgeTreeDecomp.pdf}}%
    \put(0.03158799,0.09746127){\color[rgb]{0,0,0}\makebox(0,0)[lb]{\smash{$T_2^1$}}}%
    \put(0.02991666,0.03896565){\color[rgb]{0,0,0}\makebox(0,0)[lb]{\smash{$T_2^2$}}}%
    \put(0.09927646,0.04063685){\color[rgb]{0,0,0}\makebox(0,0)[lb]{\smash{$T_3^2$}}}%
    \put(0.09760514,0.09662582){\color[rgb]{0,0,0}\makebox(0,0)[lb]{\smash{$T_3^1$}}}%
    \put(0,0){\includegraphics[width=\unitlength,page=4]{2BridgeTreeDecomp.pdf}}%
  \end{picture}%
\endgroup%

%% file: TWandVolume.bbl
\providecommand{\bysame}{\leavevmode\hbox to3em{\hrulefill}\thinspace}
\providecommand{\MR}{\relax\ifhmode\unskip\space\fi MR }
\providecommand{\MRhref}[2]{%
  \href{http://www.ams.org/mathscinet-getitem?mr=#1}{#2}
}
\providecommand{\href}[2]{#2}
\begin{thebibliography}{10}

\bibitem{Agol:Small}
Ian Agol, \emph{Small 3-manifolds of large genus}, Geom. Dedicata \textbf{102}
  (2003), 53--64. \MR{2026837}

\bibitem{Arnborg:1987:CFE:37170.37183}
Stefan Arnborg, Derek~G. Corneil, and Andrzej Proskurowski, \emph{Complexity of
  finding embeddings in a k-tree}, SIAM J. Algebraic Discrete Methods
  \textbf{8} (1987), no.~2, 277--284.

\bibitem{ARNBORG19901}
Stefan Arnborg, Andrzej Proskurowski, and Derek~G. Corneil, \emph{Forbidden
  minors characterization of partial 3-trees}, Discrete Mathematics \textbf{80}
  (1990), no.~1, 1 -- 19.

\bibitem{BenedettiPetronioHyperbolicGeom}
Riccardo Benedetti and Carlo Petronio, \emph{Lectures on hyperbolic geometry},
  Springer, 1992.

\bibitem{DBLP:journals/jct/Bienstock90}
Daniel Bienstock, \emph{On embedding graphs in trees}, J. Comb. Theory, Ser.
  {B} \textbf{49} (1990), no.~1, 103--136.

\bibitem{DBLP:journals/siamcomp/Bodlaender96}
Hans~L. Bodlaender, \emph{A linear-time algorithm for finding
  tree-decompositions of small treewidth}, {SIAM} J. Comput. \textbf{25}
  (1996), no.~6, 1305--1317.

\bibitem{BurdeZieschang}
Gerhard Burde and Heiner Zieschang, \emph{Knots}, de Gruyter Studies in
  Mathematics, vol.~5, Walter de Gruyter \& Co., Berlin, 1985. \MR{808776
  (87b:57004)}

\bibitem{burton04-regina}
Benjamin~A. Burton, \emph{Introducing {R}egina, the 3-manifold topology
  software}, Experiment. Math. \textbf{13} (2004), no.~3, 267--272.
  \MR{2103324}

\bibitem{Burton:Crushing}
\bysame, \emph{A new approach to crushing 3-manifold triangulations}, Discrete
  Comput. Geom. \textbf{52} (2014), no.~1, 116--139. \MR{3231034}

\bibitem{DBLP:journals/jct/BurtonD17}
Benjamin~A. Burton and Rodney~G. Downey, \emph{Courcelle's theorem for
  triangulations}, J. Comb. Theory, Ser. {A} \textbf{146} (2017), 264--294.

\bibitem{DBLP:conf/icalp/BurtonMS15}
Benjamin~A. Burton, Cl{\'{e}}ment Maria, and Jonathan Spreer, \emph{Algorithms
  and complexity for {T}uraev-{V}iro invariants}, Automata, Languages, and
  Programming - 42nd International Colloquium, {ICALP} 2015, Kyoto, Japan, July
  6-10, 2015, Proceedings, Part {I} (Magn{\'{u}}s~M. Halld{\'{o}}rsson, Kazuo
  Iwama, Naoki Kobayashi, and Bettina Speckmann, eds.), Lecture Notes in
  Computer Science, vol. 9134, Springer, 2015, pp.~281--293.

\bibitem{DBLP:conf/soda/BurtonS13}
Benjamin~A. Burton and Jonathan Spreer, \emph{The complexity of detecting taut
  angle structures on triangulations}, Proceedings of the Twenty-Fourth Annual
  {ACM-SIAM} Symposium on Discrete Algorithms, {SODA} 2013, New Orleans,
  Louisiana, USA, January 6-8, 2013, 2013, pp.~168--183.

\bibitem{Courcelle}
Bruno Courcelle, \emph{The monadic second-order logic of graphs. {I}.
  {R}ecognizable sets of finite graphs}, Inform. and Comput. \textbf{85}
  (1990), no.~1, 12--75. \MR{1042649}

\bibitem{snappy}
M~Culler, N~M Dunfield, and J~R Weeks, \emph{{SnapPy}, a computer program for
  studying the geometry and topology of 3-manifolds},
  \texttt{http://\allowbreak snappy.\allowbreak computop.\allowbreak org/},
  1991--2016.

\bibitem{edelsbrunner_2001}
Herbert Edelsbrunner, \emph{Geometry and topology for mesh generation},
  Cambridge Monographs on Applied and Computational Mathematics, Cambridge
  University Press, 2001.

\bibitem{FenchelHyperbolicGeom}
Werner Fenchel, \emph{Elementary geometry in hyperbolic space}, De Gruyter,
  2011.

\bibitem{FKP:DFVolJP}
David Futer, Efstratia Kalfagianni, and Jessica~S. Purcell, \emph{Dehn filling,
  volume, and the {J}ones polynomial}, J. Differential Geom. \textbf{78}
  (2008), no.~3, 429--464. \MR{MR2396249 (2009c:57010)}

\bibitem{GueritaudFuter:2bridge}
Fran{\c{c}}ois Gu{\'e}ritaud, \emph{On canonical triangulations of
  once-punctured torus bundles and two-bridge link complements}, Geom. Topol.
  \textbf{10} (2006), 1239--1284, With an appendix by David Futer. \MR{2255497}

\bibitem{HuszarSpreerWagner}
Krist{\'o}f Husz{\'a}r, Jonathan Spreer, and Uli Wagner, \emph{On the treewidth
  of triangulated 3-manifolds}, arXiv:1712.00434, 2017.

\bibitem{JacoRubinstein:LayeredTriang}
William Jaco and J.~Hyam Rubinstein, \emph{Layered-triangulations of
  3-manifolds}, \url{http://arxiv.org/abs/math/0603601}.

\bibitem{JacoRubinstein:0Eff}
\bysame, \emph{{$0$}-efficient triangulations of 3-manifolds}, J. Differential
  Geom. \textbf{65} (2003), no.~1, 61--168. \MR{2057531}

\bibitem{JacoSedgwick}
William Jaco and Eric Sedgwick, \emph{Decision problems in the space of {D}ehn
  fillings}, Topology \textbf{42} (2003), no.~4, 845--906. \MR{1958532}

\bibitem{KazhdanMargulis}
D.~A. Ka{\v{z}}dan and G.~A. Margulis, \emph{A proof of {S}elberg's
  hypothesis}, Mat. Sb. (N.S.) \textbf{75 (117)} (1968), 163--168. \MR{0223487}

\bibitem{kleiner08-perelman}
Bruce Kleiner and John Lott, \emph{Notes on {P}erelman's papers}, Geom. Topol.
  \textbf{12} (2008), no.~5, 2587--2855.

\bibitem{KobayashiRieck}
Tsuyoshi Kobayashi and Yo'av Rieck, \emph{A linear bound on the tetrahedral
  number of manifolds of bounded volume (after {J}\o rgensen and {T}hurston)},
  Topology and geometry in dimension three, Contemp. Math., vol. 560, Amer.
  Math. Soc., Providence, RI, 2011, pp.~27--42. \MR{2866921}

\bibitem{DBLP:conf/soda/MariaS17}
Cl{\'{e}}ment Maria and Jonathan Spreer, \emph{A polynomial time algorithm to
  compute quantum invariants of 3-manifolds with bounded first betti number},
  Proceedings of the Twenty-Eighth Annual {ACM-SIAM} Symposium on Discrete
  Algorithms, {SODA} 2017, Barcelona, Spain, Hotel Porta Fira, January 16-19
  (Philip~N. Klein, ed.), {SIAM}, 2017, pp.~2721--2732.

\bibitem{Mostow}
G.~D. Mostow, \emph{Strong rigidity of locally symmetric spaces}, Princeton
  University Press, Princeton, N.J., 1973, Annals of Mathematics Studies, No.
  78.

\bibitem{perelman02}
Grisha Perelman, \emph{{The entropy formula for the Ricci flow and its
  geometric applications}}, 2002.

\bibitem{perelman03}
\bysame, \emph{{Ricci flow with surgery on three--manifolds}}, 2003.

\bibitem{Purcell:book}
Jessica~S. Purcell, \emph{Hyperbolic knot theory}, Available at
  \url{http://users.monash.edu/~jpurcell/hypknottheory.html}, 2018.

\bibitem{robertson86-algorithmic}
Neil Robertson and P.~D. Seymour, \emph{Graph minors. {II}. {A}lgorithmic
  aspects of tree-width}, J. Algorithms \textbf{7} (1986), no.~3, 309--322.

\bibitem{SakumaWeeks}
Makoto Sakuma and Jeffrey Weeks, \emph{Examples of canonical decompositions of
  hyperbolic link complements}, Japan. J. Math. (N.S.) \textbf{21} (1995),
  no.~2, 393--439. \MR{1364387}

\bibitem{doi:10.1002/net.3230200304}
A.~Satyanarayana and L.~Tung, \emph{A characterization of partial 3-trees},
  Networks \textbf{20} (1990), no.~3, 299--322.

\bibitem{Seymour1994}
P.~D. Seymour and R.~Thomas, \emph{Call routing and the ratcatcher},
  Combinatorica \textbf{14} (1994), no.~2, 217--241.

\bibitem{10.1007/3-540-40996-3_17}
Dimitrios~M. Thilikos, Maria~J. Serna, and Hans~L. Bodlaender,
  \emph{Constructive linear time algorithms for small cutwidth and
  carving-width}, Algorithms and Computation (Berlin, Heidelberg) (Gerhard
  Goos, Juris Hartmanis, Jan van Leeuwen, D.~T. Lee, and Shang-Hua Teng, eds.),
  Springer Berlin Heidelberg, 2000, pp.~192--203.

\bibitem{Thurston:notes}
William Thurston, \emph{The geometry and topology of three-manifolds}, 1979,
  {\tt http://\allowbreak www.msri.org/\allowbreak gt3m/}.

\end{thebibliography}
